\documentclass{amsart}
\usepackage{amsmath}
\usepackage{amsfonts}

\theoremstyle{plain}
\newtheorem{theorem}{Theorem}
\newtheorem{lemma}{Lemma}
\newtheorem{corollary}{Corollary}
\newtheorem{proposition}{Proposition}

\theoremstyle{definition}

\theoremstyle{remark}
\newtheorem{remark}{Remark}

\numberwithin{equation}{section}
\input{tcilatex}

\begin{document}
\title[$q$-Hermite Polynomials]{On the $q-$Hermite polynomials and their
relationship with some other families of orthogonal polynomials}
\author{Pawe\l\ J. Szab\l owski}
\address{Department of Mathematics and Information Sciences, \\
Warsaw University of Technology\\
ul. Koszykowa 75, 00-662 Warsaw, Poland}
\email{pawel.szablowski@gmail.com}
\date{December 8, 2010}

\begin{abstract}
We review properties of the $q-$Hermite polynomials and indicate their links
with the Chebyshev, Rogers--Szeg\"{o}, Al-Salam--Chihara, continuous $q-$%
utraspherical polynomials. In particular we recall the connection
coefficients between these families of polynomials. We also present some
useful and important finite and infinite expansions involving polynomials of
these families including symmetric and non-symmetric kernels. In the paper
we collect scattered throughout literature useful but not widely known facts
concerning these polynomials. It is based on 43 positions of predominantly
recent literature.
\end{abstract}

\maketitle

\section{Introduction}

The aim of this paper is to review basic properties of the $q-$Hermite
polynomials and collect their not always widely known properties scattered
throughout the recent literature. The $q-$Hermite polynomials constitute a $%
1-$parameter family of orthogonal polynomials that for $q\allowbreak
=\allowbreak 1$ are equal to the well known Hermite polynomials, more
precisely the probabilistic Hermite polynomials i.e. orthogonal with respect
to the density of $N(0,1)$ distributions ($\exp (-x^{2}/2)/\sqrt{2\pi }$).
For $q\allowbreak =\allowbreak 0$ they are equal to the re-scaled Chebyshev
polynomials of the second kind, again more precisely, polynomials orthogonal
with respect to the Wigner measure i.e. the one with the density $2\sqrt{%
4-x^{2}}/\pi .$ On the other hand these polynomials are related to the so
called Rogers-Szeg\"{o} polynomials or Rogers (continuous $q-$%
ultraspherical) polynomials and other important families of polynomials such
as the Al-Salam--Chihara polynomials.

Why these polynomials are important? For one thing they are very simple and
as it will be shown in the sequel many more complicated (i.e. having more
parameters) families of orthogonal polynomials can be expressed as linear
combinations of the $q-$Hermite polynomials. Secondly since they are simple
many facts concerning them are known. Thirdly during last 15 years they
appeared within several interesting applications that came from the theories
quite distant from the classical $q-$series theory or combinatorics.

They appeared long time ago by the end of XIX-th century as a version of the
Rogers polynomials (see \cite{Rogers1}, \cite{Rogers2}, \cite{Rogers3}),
their important properties were examined by Szeg\"{o} \cite{Szego1} and
Carlitz \cite{Carlitz56}, \cite{Carlitz57}, \cite{Carlitz72} through XX-th
century, but only recently it appeared that they are important in
non-commutative probability (see e.g. \cite{Bo}, \cite{Voi2},\cite{Voi}),
quantum physics (see e.g. \cite{Flo97} \cite{Floreanini97}), combinatorics
(see e.g. \cite{ISV87}, \cite{Stanton08}, \cite{IS03}) and last but not
least ordinary, classical probability theory (see e.g. \cite{bryc1}, \cite%
{bms}, \cite{BryMaWe}, \cite{BryWes10}, \cite{matszab}, \cite{matszab2})
extending the spectrum of known, finite support measures.

To define these polynomials and briefly describe their properties one has to
adopt notation used in the so called $q$-series theory. Moreover the
terminology concerning these polynomials is not fixed and under the same
name appear sometimes different, but related to one another families of
polynomials. Thus one has to be aware of these differences.

That is why the next section of the paper will be devoted to notation,
definitions and discussion of different families of polynomials that
function under the same name. The following section will be dedicated to the
different 'finite expansions' formulae establishing relationships between
these families of polynomials including listing known the so called
'connection coefficients' and 'linearization' formulae. The last section is
dedicated to some infinite expansions involving discussed polynomials. It
consists of three subsections the first of which is devoted to different
generalizations of the Mehler expansion formula, the second one to some
useful infinite expansions including reciprocals of some kernels that have
auxiliary meaning. Finally the third subsection is dedicated to an attempt
of generalization of the $3$- dimensional Kibble--Slepian formula with the
Hermite polynomials replaced by the $q-$Hermite ones..

\section{Notation and definitions}

$q$ is a parameter. We will assume that $-1<q\leq 1$ unless otherwise
stated. The case $q\allowbreak =\allowbreak 1$ may not always be considered
directly but sometimes as left hand side limit ( i.e. $q\longrightarrow
1^{-} $). We will point out these cases.

We will use traditional notation of the $q-$series theory i.e. $\left[ 0%
\right] _{q}\allowbreak =\allowbreak 0;$ $\left[ n\right] _{q}\allowbreak
=\allowbreak 1+q+\ldots +q^{n-1}\allowbreak ,$ $\left[ n\right]
_{q}!\allowbreak =\allowbreak \prod_{j=1}^{n}\left[ j\right] _{q},$ with $%
\left[ 0\right] _{q}!\allowbreak =1,$%
\begin{equation*}
\QATOPD[ ] {n}{k}_{q}\allowbreak =\allowbreak \left\{ 
\begin{array}{ccc}
\frac{\left[ n\right] _{q}!}{\left[ n-k\right] _{q}!\left[ k\right] _{q}!} & 
, & n\geq k\geq 0 \\ 
0 & , & otherwise%
\end{array}%
\right. .
\end{equation*}%
$\binom{n}{k}$ will denote ordinary, well known binomial coefficient. 
\newline
It is useful to use the so called $q-$Pochhammer symbol for $n\geq 1:$%
\begin{equation*}
\left( a;q\right) _{n}=\prod_{j=0}^{n-1}\left( 1-aq^{j}\right) ,~~\left(
a_{1},a_{2},\ldots ,a_{k};q\right) _{n}\allowbreak =\allowbreak
\prod_{j=1}^{k}\left( a_{j};q\right) _{n}.
\end{equation*}%
with $\left( a;q\right) _{0}\allowbreak =\allowbreak 1$.

Often $\left( a;q\right) _{n}$ as well as $\left( a_{1},a_{2},\ldots
,a_{k}|q\right) _{n}$ will be abbreviated to $\left( a\right) _{n}$ and $%
\left( a_{1},a_{2},\ldots ,a_{k}\right) _{n},$ if it will not cause
misunderstanding.

It is easy to notice that $\left( q\right) _{n}=\left( 1-q\right) ^{n}\left[
n\right] _{q}!$ and that%
\begin{equation*}
\QATOPD[ ] {n}{k}_{q}\allowbreak =\allowbreak \allowbreak \left\{ 
\begin{array}{ccc}
\frac{\left( q\right) _{n}}{\left( q\right) _{n-k}\left( q\right) _{k}} & ,
& n\geq k\geq 0 \\ 
0 & , & otherwise%
\end{array}%
\right. .
\end{equation*}
\newline
The above mentioned formula is just an example where direct setting $%
q\allowbreak =\allowbreak 1$ is senseless however passage to the limit $%
q\longrightarrow 1^{-}$ makes sense.

Notice that in particular $\left[ n\right] _{1}\allowbreak =\allowbreak n,%
\left[ n\right] _{1}!\allowbreak =\allowbreak n!,$ $\QATOPD[ ] {n}{k}%
_{1}\allowbreak =\allowbreak \binom{n}{k},$ $(a)_{1}\allowbreak =\allowbreak
1-a,$ $\left( a;1\right) _{n}\allowbreak =\allowbreak \left( 1-a\right) ^{n}$
and $\left[ n\right] _{0}\allowbreak =\allowbreak \left\{ 
\begin{array}{ccc}
1 & if & n\geq 1 \\ 
0 & if & n=0%
\end{array}%
\right. ,$ $\left[ n\right] _{0}!\allowbreak =\allowbreak 1,$ $\QATOPD[ ] {n%
}{k}_{0}\allowbreak =\allowbreak 1,$ $\left( a;0\right) _{n}\allowbreak
=\allowbreak \left\{ 
\begin{array}{ccc}
1 & if & n=0 \\ 
1-a & if & n\geq 1%
\end{array}%
\right. .$

$i$ will denote imaginary unit, unless otherwise clearly stated.

In the sequel we shall also use the following useful notation:%
\begin{equation*}
S\left( q\right) =\left\{ 
\begin{array}{ccc}
\lbrack -\frac{2}{\sqrt{1-q}},\frac{2}{\sqrt{1-q}}] & if & \left\vert
q\right\vert <1 \\ 
\mathbb{R} & if & q=1%
\end{array}%
\right. .
\end{equation*}

Sometimes we will define sets of polynomials only on bounded intervals. Of
course they can be naturally extended to whole real line.

Basically each considered family of polynomials will be of one of two kinds.
The first 'kind' will be orthogonal on $S\left( q\right) $ and their names
will generally start with the capitals. The polynomials of the second 'kind'
will be orthogonal on $[-1,1]$ and their names will generally will start
with the lower case letters. There will be in fact 4 exceptions for the two
kinds of Chebyshev polynomials (traditionally denoted by $T$ and $U$)$,$ the
so called Rogers or continuous $q$-ultraspherical polynomials traditionally
denoted by $C$ and the Al-Salam--Chihara polynomials in its 'lower case
version' version traditionally denoted by $Q.$ These polynomials are
orthogonal on $[-1,1]$ but their names as mentioned before traditionally
start with the capital letter. The difference between those two kinds of
polynomials are minor. Besides the 'lower case letters' polynomials either
do not allow the case $q\allowbreak =\allowbreak 1$ or this case leads to
some trivialities. On the contrary for the 'upper case letters' polynomials
the case $q\allowbreak =\allowbreak 1$ either considered directly or
understood as a limit $q\longrightarrow 1^{-}$ leads to important
simplifications and supports intuition.

As of now it seems that the 'upper case' polynomials are more important in
the applications that appeared recently such as probability theory both
commutative and non-commutative or quantum physics, while the 'lower case'
polynomials are more typical in the special functions theory and the
combinatorics.

In brief description of certain functions, given by infinite products,
important for the discussed families of polynomials we will use the
following families of auxiliary polynomials of degree at most $2.$

In fact they are again of two different forms (as are the families of
polynomials) that are connected with the fact if considered polynomials are
orthogonal on $\left[ -1,1\right] $ regardless of $q$ or on $S\left(
q\right) .$ As the families of polynomials these auxiliary polynomials will
be denoted by the name starting with the capital if the case concerns
orthogonality on $S\left( q\right) .$

Hence we will consider for $k\geq 0:$%
\begin{eqnarray*}
v\left( x|a\right) &=&1-2ax+a^{2},~V_{q}(x|a)=1-(1-q)ax+(1-q)a^{2} \\
l\left( x|a\right) &=&(1+a)^{2}-4xa^{2},~L_{q}(x|a)=(1+a)^{2}-(1-q)x^{2}a^{2}
\\
w\left( x,y|t\right) &=&(1-t^{2})^{2}-4xyt(1+t^{2})+4t^{2}(x^{2}+y^{2}), \\
W_{q}(x,y|t) &=&(1-t^{2})^{2}-(1-q)xyt(1+t^{2})+(1-q)t^{2}(x^{2}+y^{2}).
\end{eqnarray*}

Let us notice that%
\begin{eqnarray*}
w\left( x,x|t\right) &=&(1-t)^{2}l\left( x|t\right) , \\
W_{q}\left( x,x|t\right) &=&(1-t)^{2}L_{q}\left( x|t\right) .
\end{eqnarray*}%
and that also%
\begin{eqnarray*}
(ae^{i\theta },ae^{-i\theta })_{1} &=&v\left( x|a\right) , \\
(ae^{i\theta +i\eta )},ae^{-i\theta +i\eta },ae^{i\theta -i\eta
},ae^{-i\theta -i\eta })_{1} &=&w\left( x,y|a\right) , \\
(ae^{i2\theta },ae^{-i2\theta })_{1} &=&l\left( x|a\right) ,
\end{eqnarray*}%
\begin{eqnarray}
\left( ae^{i\theta },ae^{-i\theta }\right) _{\infty } &=&\prod_{k=0}^{\infty
}v\left( x|aq^{k}\right) ,  \label{rozklv} \\
\left( te^{i\left( \theta +\phi \right) },te^{i\left( \theta -\phi \right)
},te^{-i\left( \theta -\phi \right) },te^{-i\left( \theta +\phi \right)
}\right) _{\infty } &=&\prod_{k=0}^{\infty }w\left( x,y|tq^{k}\right) ,
\label{rozklw} \\
\left( ae^{2i\theta },ae^{-2i\theta }\right) _{\infty }
&=&\prod_{k=0}^{\infty }l\left( x|aq^{k}\right) ,  \label{rozkll}
\end{eqnarray}%
where, as usually in the $q-$series theory, $x\allowbreak =\allowbreak \cos
\theta $ and $y=\cos \phi .$

The following convention will help in ordered listing of the properties of
the discussed families of polynomials. Namely the family of polynomials
whose names start with say a letter $A$ will be referred to as $A$
(similarly for the lower case $a)$. There will be one exception namely
members of the so called family of big $q$-Hermite polynomials are
traditionally denoted by letter $H$ (or $h)$ as members of the family of $q-$%
Hermite polynomials. So family of big $q-$Hermite polynomials will be
referred to by $bH.$

Let us also define the following sets of polynomials and present their
generating functions and measures with respect to which these polynomials
are orthogonal if these measures are positive.

\subsection{Hermite}

The Hermite polynomials are defined by the following $3-$term recurrence (%
\ref{_1}), below: 
\begin{equation}
xH_{n}\left( x\right) =H_{n+1}\left( x\right) +nH_{n-1},  \label{_1}
\end{equation}%
with $H_{0}\left( x\right) \allowbreak =\allowbreak H_{1}\left( x\right)
\allowbreak =\allowbreak 1$. They slightly differ from the Hermite
polynomials $h_{n}$ considered in most of the books on special functions.
Namely 
\begin{equation*}
2xh_{n}\left( x\right) =h_{n+1}(x)+2nh_{n-1}(x),
\end{equation*}%
with $h_{-1}\left( x\right) =0,$ $h_{0}\left( x\right) \allowbreak
=\allowbreak 1.$

It is known that polynomials $\left\{ h_{n}\right\} $ are orthogonal with
respect to $\exp \left( -x^{2}\right) $ while $\left\{ H_{n}\right\} $ with
respect to $\exp \left( -x^{2}/2\right) .$ Moreover $H_{n}\left( x\right)
=h_{n}\left( x/\sqrt{2}\right) /\left( \sqrt{2}\right) ^{n}.$ Besides we
have 
\begin{eqnarray}
\exp \left( xt-t^{2}/2\right) \allowbreak &=&\allowbreak \sum_{k\geq 0}\frac{%
t^{k}}{k!}H_{k}\left( x\right) ,  \label{_gH} \\
\exp \left( 2xt-t^{2}\right) \allowbreak &=&\allowbreak \sum_{k\geq 0}\frac{%
t^{k}}{k!}h_{k}\left( x\right) .  \label{_gh}
\end{eqnarray}

\subsection{Chebyshev}

They are of two kinds. The Chebyshev polynomials of the first kind $\left\{
T_{n}\right\} _{n\geq -1}$ are defined by the following $3-$term recursion%
\begin{equation}
2xT_{n}\left( x\right) \allowbreak =\allowbreak T_{n+1}\left( x\right)
+T_{n-1}\left( x\right) ,  \label{czeb}
\end{equation}%
for $n\geq 1,$ with $T_{0}\left( x\right) \allowbreak =\allowbreak 1,$ $%
T_{1}\left( x\right) \allowbreak =\allowbreak x.$ One can define them also
in the following way:%
\begin{equation}
T_{n}\left( \cos \theta \right) =\cos \left( n\theta \right) .  \label{repCh}
\end{equation}

The Chebyshev polynomials $\left\{ U_{n}\left( x\right) \right\} _{n\geq 0}$
of the second kind are defined by the same $3$-term recurrence i.e. (\ref%
{czeb}) with the different initial conditions, namely $U_{0}\left( x\right)
\allowbreak =\allowbreak 1$ and $U_{1}\left( x\right) \allowbreak
=\allowbreak 2x.$ One shows that they can be defined also by 
\begin{equation}
U_{n}\left( \cos \theta \right) \allowbreak =\allowbreak \frac{\sin \left(
n+1\right) \theta }{\sin \theta }.  \label{repCh2}
\end{equation}

We have 
\begin{eqnarray*}
\int_{-1}^{1}T_{n}\left( x\right) T_{m}\left( x\right) \frac{dx}{\pi \sqrt{%
1-x^{2}}} &=&\left\{ 
\begin{array}{ccc}
1 & if & m=n=0 \\ 
1/2 & if & m=n\neq 0 \\ 
0 & if & m\neq n%
\end{array}%
\right. , \\
\int_{-1}^{1}U_{n}\left( x\right) U_{m}\left( x\right) \frac{2\sqrt{1-x^{2}}%
}{\pi }dx &=&\left\{ 
\begin{array}{ccc}
1 & if & m=n \\ 
0 & if & m\neq n%
\end{array}%
\right. ,
\end{eqnarray*}%
and for $\left\vert t\right\vert \leq 1$ 
\begin{eqnarray}
\sum_{k=0}^{\infty }t^{k}T_{k}\left( x\right) &=&\frac{1-tx}{1-2tx+t^{2}},
\label{_gT} \\
\sum_{k=0}^{\infty }t^{k}U_{k}\left( x\right) &=&\frac{1}{1-2tx+t^{2}}.
\label{_gU}
\end{eqnarray}

\subsection{q-Hermite}

The $q-$Hermite polynomials are defined by: 
\begin{equation}
2xh_{n}(x|q)=h_{n+1}(x|q)+(1-q^{n})h_{n-1}(x|q),  \label{q-cont}
\end{equation}%
for $n\geq 1$ with $h_{-1}(x|q)=0,$ $h_{0}(x|q)=1$.

The polynomials $h_{n}$ are often called the continuous $q-$Hermite
polynomials. Since the terminology is not fixed and also since we will
consider only two types of them (defined by (\ref{q-cont}) and (\ref{He}))
we will use the name $q-$Hermite polynomials for the brevity.

In fact we will also use the following transformed form of the polynomials $%
h_{n},$ namely the polynomials: 
\begin{equation}
H_{n}\left( x|q\right) \allowbreak =\allowbreak (1-q)^{-n/2}h_{n}\left( 
\frac{x\sqrt{1-q}}{2}|q\right) .  \label{ch_of_vars}
\end{equation}%
It is easy to notice that the polynomials $\left\{ H_{n}\left( x|q\right)
\right\} $ satisfy the following $3-$term recurrence%
\begin{equation}
xH_{n}\left( x|q\right) =H_{n+1}\left( x|q\right) \allowbreak +\left[ n%
\right] _{q}H_{n-1}\left( x\right) ,  \label{He}
\end{equation}%
for $n\geq 1$ with $H_{-1}\left( x|q\right) \allowbreak =\allowbreak 0$, $%
H_{1}\left( x|q\right) \allowbreak =\allowbreak 1$. The name is justified
since one can easily show that $n\geq -1$ 
\begin{equation*}
H_{n}\left( x|1\right) \allowbreak =\allowbreak H_{n}\left( x\right) .
\end{equation*}%
Notice that since $\left[ n\right] _{0}\allowbreak =\allowbreak 1$ for $%
n\geq -1$ we have 
\begin{equation}
H_{n}\left( x|0\right) \allowbreak =\allowbreak U_{n}\left( x/2\right) .
\label{q=0}
\end{equation}

It is known that (see e.g. \cite{KLS}(14.26.2)): 
\begin{eqnarray*}
\int_{-1}^{1}h_{n}\left( x|q\right) h_{m}\left( x|q\right) f_{h}\left(
x|q\right) dx &=&\left\{ 
\begin{array}{ccc}
\left( q\right) _{n} & if & m=n \\ 
0 & if & m\neq n%
\end{array}%
\right. , \\
\int_{S\left( q\right) }H_{n}\left( x\right) H_{m}\left( x\right)
f_{N}\left( x\right) dx &=&\left\{ 
\begin{array}{ccc}
\left[ n\right] _{q}! & if & m=n \\ 
0 & if & m\neq n%
\end{array}%
\right. ,
\end{eqnarray*}%
where we denoted 
\begin{eqnarray}
f_{h}\left( x|q\right) &=&\frac{2\left( q\right) _{\infty }\sqrt{1-x^{2}}}{%
\pi }\prod_{k=1}^{\infty }l\left( x|q^{k}\right) ,  \label{fN} \\
f_{N}\left( x|q\right) \allowbreak &=&\allowbreak \left\{ 
\begin{array}{ccc}
\sqrt{1-q}f_{h}(x\sqrt{1-q}/2|q)/2 & if & \left\vert q\right\vert <1 \\ 
\exp \left( -x^{2}/2\right) /\sqrt{2\pi } & if & q=1%
\end{array}%
\right.
\end{eqnarray}%
and (see \cite{KLS}(14.26.11))%
\begin{eqnarray}
\sum_{j=0}^{\infty }\frac{t^{j}}{\left( q\right) _{j}}h_{j}\left( x|q\right)
&=&\frac{1}{\prod_{k=0}^{\infty }v\left( x|tq^{k}\right) },  \label{_gqh} \\
\sum_{j=0}^{\infty }\frac{t^{j}}{\left[ j\right] _{q}!}H_{j}\left(
x|q\right) &=&\frac{1}{\prod_{k=0}^{\infty }V_{q}\left( x|tq^{k}\right) }.
\label{_gqH}
\end{eqnarray}

Convergence is in the above formulae for $\left\vert x|,|t\right\vert \leq 1$
in (\ref{_gqh}) and $x\in S\left( q\right) $ and $\left\vert t\sqrt{1-q}%
\right\vert \leq 1$ in (\ref{_gqH}).

One proves also that 
\begin{eqnarray}
\lim_{q->1^{-}}f_{N}\left( x|q\right) &=&\frac{1}{\sqrt{2\pi }}\exp \left( -%
\frac{x^{2}}{2}\right) ,  \label{qGauss} \\
\lim_{q->1^{-}}\frac{1}{\prod_{k=0}^{\infty }V_{q}\left( x|tq^{k}\right) }
&=&\exp \left( xt-\frac{t^{2}}{2}\right) .  \label{fgH}
\end{eqnarray}%
Rigorous and easy proofs of these facts can be found in \cite{ISV87}. The
convergence in distribution is obvious since we have $\forall n\geq 0:$ $%
\lim_{q->1^{-}}H_{n}\left( x|q\right) \allowbreak =\allowbreak H_{n}\left(
x|1\right) $ consequently we have the convergence of moments.

Let us remark that the density $f_{N}$ is a real probabilistic density i.e.
integrates to $1.$ Since we have (\ref{qGauss}) it is sometimes called $q-$%
Gaussian or $q-$Normal. It has appeared in non-commutative probability
context in an important paper \cite{Bo}. Later in classical probability
context appeared in \cite{bryc1} and \cite{bms}. Its further properties
including algorithm how to simulate sequences if independent random
observations having $f_{N}$ as its density were presented in \cite%
{Szab-qGauss}.

One considers also the small generalization of the $q-$Hermite polynomials
namely the so called big continuous $q-$Hermite polynomials. i.e. the
polynomials defined by the following 3-term recurrence:%
\begin{eqnarray*}
(2x-aq^{n})h_{n}\left( x|a,q\right) &=&h_{n+1}\left( x|a,q\right)
+(1-q^{n})h_{n-1}\left( x|a,q\right) , \\
(x-aq^{n})H_{n}\left( x|a,q\right) &=&H_{n+1}\left( x|a,q\right) +\left[ n%
\right] _{q}H_{n-1}\left( x|a,q\right) ,
\end{eqnarray*}%
with initial conditions: $h_{-1}\left( x|a,q\right) \allowbreak =\allowbreak
H_{-1}\left( x|a,q\right) \allowbreak =\allowbreak 0$ and $h_{0}\left(
x|a,q\right) \allowbreak =\allowbreak H_{0}\left( x|a,q\right) \allowbreak
=\allowbreak 1.$ For the sake of brevity we will call them simply big $q-$%
Hermite polynomials. They are obviously inter-related by 
\begin{equation*}
H_{n}\left( x|a,q\right) \allowbreak =\allowbreak h_{n}\left( \frac{\sqrt{1-q%
}x}{2}|a\sqrt{1-q},q\right) /\left( 1-q\right) ^{n/2}.
\end{equation*}%
Notice that, using well known properties of the ordinary Hermite
polynomials, we have: 
\begin{equation*}
H_{n}\left( x|a,1\right) \allowbreak =\allowbreak H_{n}\left( x-a\right) .
\end{equation*}%
One can easily show (by calculating generating function and comparing it
with (\ref{gfha}), below and then applying (\ref{rozklv})) that%
\begin{equation*}
h_{n}\left( x|a,q\right) =\sum_{k=0}^{n}\QATOPD[ ] {n}{k}_{q}\left(
ae^{i\theta }\right) _{k}e^{i\left( n-2k\right) \theta },
\end{equation*}%
where as usually $x\allowbreak =\allowbreak \cos \theta .$

We have (see e.g. \cite{KLS}(14.18.13)) 
\begin{eqnarray}
\sum_{j=0}^{\infty }\frac{t^{j}}{\left( q\right) _{j}}h_{j}\left(
x|a,q\right) &=&\frac{\left( at\right) _{\infty }}{\prod_{k=0}^{\infty
}v\left( x|tq^{k}\right) },  \label{gfha} \\
\sum_{j=0}^{\infty }\frac{t^{j}}{\left[ j\right] _{q}!}H_{j}\left(
x|a,q\right) &=&\frac{\left( (1-q)at\right) _{\infty }}{\prod_{k=0}^{\infty
}V_{q}\left( x|tq^{k}\right) }.  \label{gfHa}
\end{eqnarray}%
We have also the following orthogonality relationships for $\left\vert
a\right\vert <1$ (again e.g. \cite{KLS}(14.18.2)) 
\begin{equation*}
\int_{S\left( q\right) }H_{n}\left( x|a,q\right) H_{m}\left( x|a,q\right)
\allowbreak f_{bN}\left( x|a,q\right) =\allowbreak \left\{ 
\begin{array}{ccc}
0 & if & n\neq m \\ 
\left[ n\right] _{q}! & if & n=m%
\end{array}%
\right. ,
\end{equation*}%
where 
\begin{equation*}
f_{bN}\left( x|a,q\right) \allowbreak =\allowbreak f_{N}\left( x|q\right) 
\frac{1}{\prod_{k=0}^{\infty }V_{q}\left( x|aq^{j}\right) },
\end{equation*}%
with similar formula for the polynomials $h_{n}\left( x|a,q\right) .$

It should be mentioned also that if $a>1$ then the measure that makes these
polynomials orthogonal has apart from absolutely continuous part with the
above mentioned density also $\#\{k:1<aq^{k}<a\}$ atoms at points 
\begin{equation}
x_{k}\allowbreak =\allowbreak (aq^{k}+a^{-1}q^{-k})/2,  \label{atoms}
\end{equation}%
with weights 
\begin{equation*}
\hat{w}_{k}\allowbreak =\allowbreak \frac{(1-a^{2}q^{2k})(a^{-2})_{\infty
}(a^{2})_{k}}{(1-a^{2})(q)_{k}}q^{-(3k^{2}+k)/2}(\frac{-1}{a^{4}})^{k}.
\end{equation*}

The family of the big $q-$Hermite polynomials will be referred to by symbol $%
bH.$ For details see \cite{KLS}(14.18.3).

\subsection{Al-Salam--Chihara}

In the literature connected with the special functions as the
Al-Salam--Chihara (ASC) function polynomials defined recursively:%
\begin{equation}
(2x-(a+b)q^{n})Q_{n}\left( x|a,b,q\right) =Q_{n+1}\left( x|a,b,q\right)
\allowbreak +(1-abq^{n-1})(1-q^{n})Q_{n-1}(x|a,b,q),  \label{AlSC1}
\end{equation}%
with $Q_{-1}\left( x|a,b,q\right) \allowbreak =\allowbreak 0,$ $Q_{0}\left(
x|a,b,q\right) \allowbreak =\allowbreak 1$. From Favard's theorem (\cite{IA}%
) it follows that if $\left\vert ab\right\vert \leq 1,$ then there exists
positive measure with respect to which polynomials $Q_{n}$ are orthogonal.
Further when $\left\vert a\right\vert ,\left\vert b\right\vert <1,$ then
this measure has density.

As in the case of big $q-$Hermite polynomials if one of the parameters $a$
and $b$ is greater than $1$ then the measure that makes ASC polynomials
orthogonal has $\#\{k:1<aq^{k}<a\}$ atoms located at points $x_{k}$ defined
by (\ref{atoms}) with weights given by :%
\begin{equation*}
\hat{w}_{k}\allowbreak =\allowbreak \frac{(a^{-2})_{\infty
}(1-a^{2}q^{2k})(a^{2},ab)_{k}}{(b/a)_{\infty }(1-a^{2})(q,aq/b)_{k}}%
q^{-k^{2}}(\frac{1}{a^{3}b})^{k}.
\end{equation*}%
For details see \cite{KLS}(14.8.3).

Since we are interested in the ASC polynomials in connection with the $q-$%
Hermite polynomials we will consider only the case $\left\vert a\right\vert
,\left\vert b\right\vert <1$.

We will more often use these polynomials with new parameters $\rho $ and $y$
defined by $a\allowbreak =\allowbreak \frac{\sqrt{1-q}}{2}\rho (y\allowbreak
-\allowbreak i\sqrt{\frac{4}{1-q}-y^{2}}),b\allowbreak =\allowbreak \frac{%
\sqrt{1-q}}{2}\rho (y\allowbreak +\allowbreak i\sqrt{\frac{4}{1-q}-y^{2}}),$
such that $y^{2}\leq 4/(1-q),$ $\left\vert \rho \right\vert <1$. To support
intuition let us remark:%
\begin{equation*}
a+b=\sqrt{1-q}\rho y,~~ab=\rho ^{2}.
\end{equation*}%
More precisely we will also consider the polynomials 
\begin{equation}
P_{n}\left( x|y,\rho ,q\right) \allowbreak =Q_{n}\left( x\frac{\sqrt{1-q}}{2}%
|\frac{\rho \sqrt{1-q}}{2}(y\allowbreak -\allowbreak i\sqrt{\frac{4}{1-q}%
-y^{2}}),\frac{\rho \sqrt{1-q}}{2}(y\allowbreak +\allowbreak i\sqrt{\frac{4}{%
1-q}-y^{2}}),q\right) /(1-q)^{n/2}.  \label{podstawienie}
\end{equation}%
It is also of use to consider another version of the ASC polynomials, namely
for $\left\vert x\right\vert ,\left\vert y\right\vert ,\left\vert \rho
\right\vert ,\left\vert q\right\vert \allowbreak <\allowbreak 1:$ 
\begin{equation}
p_{n}\left( x|y,\rho ,q\right) =P_{n}\left( \frac{2x}{\sqrt{1-q}}|\frac{2y}{%
\sqrt{1-q}},\rho ,q\right) .  \label{male_p}
\end{equation}%
One can easily show that the polynomials $P_{n}$ and $p_{n}$ satisfy the
following $3-$term recurrence:

\begin{gather}
(x-\rho yq^{n})P_{n}(x|y,\rho ,q)=  \label{AlSC} \\
P_{n+1}(x|y,\rho ,q)+(1-\rho ^{2}q^{n-1})[n]_{q}P_{n-1}(x|y,\rho ,q), \\
2(x-\rho yq^{n})p_{n}\left( x|y,\rho ,q\right) =\allowbreak  \label{alsc} \\
p_{n+1}\left( x|y,\rho ,q\right) \allowbreak +\left( 1-\rho
^{2}q^{n-1}\right) (1-q^{n})p_{n-1}\left( x|y,\rho ,q\right) ,
\end{gather}%
with $P_{-1}\left( x|y,\rho ,q\right) \allowbreak =\allowbreak p_{-1}\left(
x|y,\rho ,q\right) \allowbreak =\allowbreak 0,$ $P_{0}\left( x|y,\rho
,q\right) \allowbreak =\allowbreak p_{0}\left( x|y,\rho ,q\right)
\allowbreak =\allowbreak 1$ since as stated above $a\allowbreak +\allowbreak
b\allowbreak =\allowbreak \rho y\sqrt{1-q}$ and $ab\allowbreak =\allowbreak
\rho ^{2}$ in the case of the polynomials $P$ and $a\allowbreak +\allowbreak
b\allowbreak =\allowbreak 2\rho y$ and $ab\allowbreak =\allowbreak \rho ^{2}$
in the case of the polynomials $p$.

The polynomials $\left\{ P_{n}\right\} $ have a nice probabilistic
interpretation see e.g. \cite{bms}. To support intuition let us notice that 
\begin{eqnarray*}
P_{n}\left( x|y,\rho ,1\right) &=&(1-\rho ^{2})^{n/2}H_{n}\left( \frac{%
x-\rho y}{\sqrt{1-\rho ^{2}}}\right) , \\
P_{n}\left( x|y,\rho ,0\right) &=&U_{n}\left( x/2\right) -\rho
yU_{n-1}\left( x/2\right) +\rho ^{2}U_{n-2}\left( x/2\right) ,
\end{eqnarray*}%
if we define $U_{-r}\left( x\right) =0,$ $r\geq 1.$

We have the following orthogonality relationships (see \cite{KLS}(14.8.2))
satisfied for $\left\vert a\right\vert ,\left\vert b\right\vert <1$:%
\begin{equation*}
\int_{-1}^{1}Q_{n}\left( x|a,b,q\right) Q_{m}\left( x|a,b,q\right) \omega
\left( x|a,b,q\right) dx=\left\{ 
\begin{array}{ccc}
0 & if & n\neq m \\ 
\left( q\right) _{n}\left( ab\right) _{n} & if & m=m%
\end{array}%
\right. ,
\end{equation*}%
where 
\begin{equation*}
\omega \left( x|a,b,q\right) =\frac{\left( q\right) _{\infty }\left(
ab\right) _{\infty }}{2\pi \sqrt{1-x^{2}}}\prod_{k=0}^{\infty }\frac{l\left(
x|q^{k}\right) }{%
((1-abq^{2k})^{2}-2x(a+b)q^{k}(1+abq^{2k})+q^{2k}ab(4x^{2}+(a+b)^{2}/(ab))}.
\end{equation*}%
Also after passing to the parameters $\rho $ and $y$ we get (see \cite{bms}):%
\begin{equation}
\int_{S\left( q\right) }P_{n}(x|y,\rho ,q)P_{m}\left( x|y,\rho ,q\right)
f_{CN}\left( x|y,\rho ,q\right) dx=\left\{ 
\begin{array}{ccc}
0 & if & m\neq n \\ 
\left[ n\right] _{q}!\left( \rho ^{2}\right) _{n} & if & m=n%
\end{array}%
\right. ,  \label{PnPm}
\end{equation}%
where we denoted for $\left\vert q\right\vert <1$ : 
\begin{subequations}
\begin{equation}
f_{CN}\left( x|y,\rho ,q\right) =\frac{\sqrt{1-q}\left( q\right) _{\infty
}\left( \rho ^{2}\right) _{\infty }}{2\pi \sqrt{L_{q}\left( x|1\right) }}%
\prod_{k=0}^{\infty }\frac{L_{q}\left( x|q^{k}\right) }{W_{q}\left( x,y|\rho
q^{k}\right) }.  \label{fCN}
\end{equation}%
Let us notice that : 
\end{subequations}
\begin{equation*}
f_{CN}\left( x|y,\rho ,q\right) =f_{N}\left( x|q\right) \frac{\left( \rho
^{2}\right) _{\infty }}{\prod_{k=0}^{\infty }W_{q}\left( x,y|\rho
q^{k}\right) }.
\end{equation*}%
We also set 
\begin{equation}
f_{CN}\left( x|y,\rho ,1\right) \allowbreak =\allowbreak \frac{1}{\sqrt{2\pi
\left( 1-\rho ^{2}\right) }}\exp \left( -\frac{\left( x-\rho y\right) ^{2}}{%
2\left( 1-\rho ^{2}\right) }\right) .  \label{q-CGauss}
\end{equation}%
Notice that we have also 
\begin{equation*}
f_{CN}\left( x|y,\rho ,0\right) \allowbreak =\allowbreak \frac{(1-\rho ^{2})%
\sqrt{4-x^{2}}}{2\pi W_{q}(x,y|\rho )},
\end{equation*}%
which is called Kesten--McKay density.

Again one shows (see e.g. \cite{ISV87}) that 
\begin{equation*}
f_{CN}\left( x|y,\rho ,q\right) \underset{q\rightarrow 1^{-}}{%
\longrightarrow }f_{CN}\left( x|y,\rho ,1\right) .
\end{equation*}%
Following \cite{SzablKer} we have: $\forall \left\vert q\right\vert
<1,x,y\in S\left( q\right) :$%
\begin{equation*}
0<\frac{\left( \rho ^{2}\right) _{\infty }}{\left( -\left\vert \rho
\right\vert \right) _{\infty }^{4}}\leq \frac{f_{CN}\left( x|y,\rho
,q\right) }{f_{N}\left( x|q\right) }\leq \frac{\left( \rho ^{2}\right)
_{\infty }}{\left( \left\vert \rho \right\vert \right) _{\infty }^{4}}.
\end{equation*}%
One shows (see e.g. \cite{bms})) that for $\left\vert x\right\vert $ ,$%
\left\vert z\right\vert \in S\left( q\right) :$%
\begin{equation*}
\int_{S\left( q\right) }f_{CN}\left( x|y,\rho _{1},q\right) f_{CN}\left(
y|z,\rho _{2},q\right) dy=f_{CN}\left( x|z,\rho _{1}\rho _{2},q\right) .
\end{equation*}%
This property is nothing else but Chapman--Kolmogorov property satisfied by
the density $f_{CN}$ interpreted as the density of the transition
distribution of some Markov chain.

Distribution with the density $f_{CN}$ is sometimes called conditional $q-$%
Gaussian or conditional $q-$Normal since we have (\ref{q-CGauss}). It
appeared in \cite{Bo} and later was analyzed in \cite{bryc1} and \cite{bms}.

We also have 
\begin{equation*}
\sum_{k=0}^{\infty }\frac{t^{k}}{\left( q\right) _{k}}Q_{k}\left(
x|a,b,q\right) =\frac{\left( at,bt\right) _{\infty }}{\prod_{j=0}^{\infty
}v\left( x|tq^{j}\right) },
\end{equation*}%
and for the parameters $\rho $ and $y:$%
\begin{equation*}
\sum_{k=0}^{\infty }\frac{t^{k}}{\left[ k\right] _{q}!}P_{k}\left( x|y,\rho
,q\right) =\prod_{j=0}^{\infty }\frac{V_{q}\left( y|\rho tq^{j}\right) }{%
V_{q}\left( x|tq^{j}\right) }.
\end{equation*}

\subsection{Continuous $q-$utraspherical polynomials}

It turns out that the polynomials $\left\{ H_{n}\right\} _{n\geq -1}$ are
also related to another family of orthogonal polynomials $\left\{
C_{n}\left( x|\beta ,q\right) \right\} _{n\geq -1}$ which was considered by
Rogers in 1894 (see \cite{Rogers2}). Now they are called the continuous $q-$%
utraspherical polynomials. The polynomials $C_{n}$ can be defined through
their $3-$recurrence (see \cite{KLS}(14.10.19))

\begin{equation*}
2(1-\beta q^{n})xC_{n}(x|\beta ,q)=(1-q^{n+1})C_{n+1}\left( x|\beta
,q\right) \allowbreak +(1-\beta ^{2}q^{n-1})C_{n-1}\left( x|\beta ,q\right) ,
\end{equation*}%
for $n\geq 0,$ with $C_{-1}\left( x|\beta ,q\right) \allowbreak =\allowbreak
0,$ $C_{0}\left( x|\beta ,q\right) \allowbreak =\allowbreak 1,$ where $\beta 
$ is a real parameter such that $\left\vert \beta \right\vert <1$. One shows
(see e.g. \cite{IA}(13.2.1)) that for $\left\vert q\right\vert ,\left\vert
\beta \right\vert <1,$ $\forall n\in \mathbb{N}:$ 
\begin{equation*}
C_{n}\left( x|\beta ,q\right) \allowbreak =\allowbreak \sum_{k=0}^{n}\frac{%
\left( \beta \right) _{k}\left( \beta \right) _{n-k}}{\left( q\right)
_{k}\left( q\right) _{n-k}}e^{i\left( n-2k\right) \theta },
\end{equation*}%
where $x=\cos \theta .$ Hence we have (following formula (\ref{cH})): 
\begin{equation*}
C_{n}\left( x|0,q\right) \allowbreak =\allowbreak \frac{h_{n}\left(
x|q\right) }{\left( q\right) _{n}}.
\end{equation*}%
In fact we will consider slightly modified polynomials $C_{n}.$ Namely we
will consider polynomials $R_{n}\left( x|\beta ,q\right) $ related to
polynomials $C_{n}$ through the relationship: 
\begin{equation}
C_{n}\left( x|\beta ,q\right) \allowbreak =\allowbreak \left( 1-q\right)
^{n/2}R_{n}\left( \frac{2x}{\sqrt{1-q}}|\beta ,q\right) /\left( q\right)
_{n},n\geq 1.  \label{q-US}
\end{equation}%
One can easily check that the polynomials $\left\{ R_{n}\right\} $ satisfy
the following $3-$term recurrence:%
\begin{equation}
\left( 1-\beta q^{n}\right) xR_{n}\left( x|\beta ,q\right) =R_{n+1}\left(
x|\beta ,q\right) +\left( 1-\beta ^{2}q^{n-1}\right) \left[ n\right]
_{q}R_{n-1}\left( x|\beta ,q\right) .  \label{R}
\end{equation}%
We have an easy Proposition

\begin{proposition}
\label{szczeg}For $n\geq 1:$ i) $R_{n}\left( x|0,q\right) \allowbreak
=\allowbreak H_{n}\left( x|q\right) ,$

ii) $R_{n}\left( x|q,q\right) \allowbreak =\allowbreak \left( q\right)
_{n}U_{n}\left( x\sqrt{1-q}/2\right) ,$

iii) $\lim_{\beta ->1^{-}}\frac{R_{n}\left( x|\beta ,q\right) }{\left( \beta
\right) _{n}}\allowbreak =\allowbreak 2\frac{T_{n}\left( x\sqrt{1-q}%
/2\right) }{(1-q)^{n/2}}.$
\end{proposition}

\begin{proof}
i) direct calculation. ii) We have for $\beta \allowbreak =\allowbreak q$ : $%
\tilde{R}_{n+1}\left( x|q,q\right) \allowbreak =\allowbreak x\tilde{R}%
_{n}\left( x|q,q\right) -\tilde{R}_{n-1}\left( x|q,q\right) ,$ where we
denoted $\tilde{R}_{n}(x|q,q)\allowbreak =\allowbreak (1-q)^{n/2}R_{n}\left(
x|q,q\right) /\left( q\right) _{n}.$ So visibly since $R_{0}\left(
x|q,q\right) \allowbreak =\allowbreak 1$ and $R_{1}\left( x|q,q\right)
\allowbreak =\allowbreak x.$ iii) Let us first denote $F_{n}\left( x|\beta
,q\right) \allowbreak =\allowbreak \frac{R_{n}\left( x|\beta ,q\right) }{%
\left( \beta \right) _{n}},$ write the $3\allowbreak -$term recurrence for
it obtaining 
\begin{equation*}
F_{n+1}\left( x|\beta ,q\right) \allowbreak =\allowbreak xF_{n}\left(
x|\beta ,q\right) \allowbreak -\allowbreak \frac{(1-q^{n})(1-\beta
^{2}q^{n-1})}{(1-q)(1-\beta q^{n})(1-\beta q^{n-1})}F_{n-1}\left( x|\beta
,q\right) ,
\end{equation*}
with $F_{-1}\left( x|\beta ,q\right) \allowbreak =\allowbreak 0,$ $%
F_{0}\left( x|\beta ,q\right) \allowbreak =\allowbreak 1$ and let $\beta
->1^{-}.$ We immediately see that the limit, denote it by $F_{n}\left(
x|1,q\right) ,$ satisfies the following the $3-$term recurrence: 
\begin{equation*}
F_{n+1}\left( x|1,q\right) \allowbreak =\allowbreak xF_{n}\left(
x|1,q\right) \allowbreak -\allowbreak \frac{F_{n-1}\left( x|1,q\right) }{%
(1-q)},
\end{equation*}
which confronted with the $3-$term recurrence satisfied by the polynomials $%
T_{n}$ proves our assertion.
\end{proof}

It is known that (see e.g. \cite{IA}(13.2.4)): 
\begin{eqnarray*}
\int_{-1}^{1}C_{n}\left( x|\beta ,q\right) C_{m}\left( x|\beta ,q\right)
f_{C}\left( x|\beta ,q\right) dx &=&\left\{ 
\begin{array}{ccc}
0 & if & m\neq n \\ 
\frac{\left( \beta ^{2}\right) _{n}}{\left( 1-\beta q^{n}\right) \left(
q\right) _{n}} & if & m=n%
\end{array}%
\right. , \\
\int_{S\left( q\right) }R_{n}\left( x|\beta ,q\right) R_{m}\left( x|\beta
,q\right) f_{R}\left( x|\beta ,q\right)  &=&\left\{ 
\begin{array}{ccc}
0 & when & n\neq m \\ 
\frac{\left( 1-\beta \right) \left( \beta ^{2}\right) _{n}[n]_{q}!}{\left(
1-\beta q^{n}\right) } & when & n=m%
\end{array}%
\right. ,
\end{eqnarray*}

where we denoted 
\begin{eqnarray}
f_{C}(x|\beta ,q)\allowbreak &=&\allowbreak \frac{(\beta ^{2})_{\infty }}{%
(1-\beta )(\beta ,\beta q)_{\infty }}f_{h}\left( x|q\right)
/\prod_{j=1}^{\infty }l\left( x|\beta q^{j}\right) ,  \label{fR} \\
f_{R}\left( x|\beta ,q\right) &=&\sqrt{1-q}f_{C}(x\sqrt{1-q}/2|q)/2 \\
&=&\frac{\left( q,\beta ^{2}\right) _{\infty }\sqrt{1-q}}{\left( \beta
,\beta q\right) _{\infty }2\pi \sqrt{L_{q}\left( x|1\right) }}%
\prod_{k=0}^{\infty }\frac{L_{q}\left( x|q^{k}\right) }{L_{q}\left( x|\beta
q^{k}\right) }.
\end{eqnarray}%
Let us remark that 
\begin{equation*}
f_{R}\left( x|\beta ,q\right) =f_{N}\left( x|q\right) \frac{\left( \beta
^{2}\right) _{\infty }}{\left( \beta ,\beta q\right) _{\infty
}\prod_{k=0}^{\infty }L_{q}\left( x|\beta q^{k}\right) }.
\end{equation*}%
Notice also that examining the 3-term recurrence satisfied by $P_{n}$ and $%
R_{n}$we see $\forall n\geq -1:$ 
\begin{equation*}
P_{n}\left( x|x,\rho ,q\right) =R_{n}\left( x|\rho ,q\right) ,
\end{equation*}%
and that for $\left\vert x\right\vert ,\left\vert y\right\vert \in S\left(
q\right) $%
\begin{equation*}
f_{CN}\left( x|x,\rho ,q\right) /(1-\rho )\allowbreak =\allowbreak
f_{R}\left( x|\rho ,q\right) ,
\end{equation*}%
since we have $(1-\rho ^{2}q^{2k})^{2}\allowbreak -\allowbreak (1-q)\rho
q^{k}(1+\rho ^{2}q^{2k})x^{2}\allowbreak +\allowbreak 2(1-q)\rho
^{2}x^{2}q^{2k}\allowbreak =\allowbreak \left( 1-\rho q^{k}\right)
^{2}(\left( 1+\rho q^{k}\right) ^{2}\allowbreak -\allowbreak (1-q)\rho
x^{2}q^{k})$ and the fact that$\frac{\left( \rho \right) _{\infty }\left(
\rho q\right) _{\infty }}{\left( \rho \right) _{\infty }^{2}}\allowbreak
=\allowbreak \frac{1}{1-\rho }.$

We also have 
\begin{eqnarray}
\sum_{k=0}^{\infty }t^{k}C_{k}\left( x|\beta ,q\right)
&=&\prod_{k=0}^{\infty }\frac{v\left( x|\beta tq^{k}\right) }{v\left(
x|tq^{k}\right) },  \label{fgC} \\
\sum_{k=0}^{\infty }\frac{t^{k}}{\left[ k\right] _{q}!}R_{k}\left( x|\beta
,q\right) &=&\prod_{j=0}^{\infty }\frac{V_{q}\left( x|\beta tq^{k}\right) }{%
V_{q}\left( x|q^{k}t\right) }.  \label{fgR}
\end{eqnarray}

\begin{remark}
Assertion ii) of Proposition \ref{szczeg} could have been deduced also from (%
\ref{fgR}), namely putting $\beta \allowbreak =\allowbreak q$ we get $%
\sum_{k=0}^{\infty }\frac{t^{k}}{\left[ k\right] _{q}!}R_{k}\left(
x|q,q\right) \allowbreak =\allowbreak \frac{1}{(1-(1-q)tx+(1-q)t^{2})}$which
confronted with (\ref{_gU}) and formula $\left( q\right) _{n}=\left(
1-q\right) ^{n}\left[ n\right] _{q}!$ leads to the conclusion that $%
R_{k}\left( x|q,q\right) /\left( q\right) _{k}\allowbreak =\allowbreak
U_{k}\left( x\sqrt{1-q}/2\right) $.$\allowbreak $ Following this idea we see
that 
\begin{equation*}
\sum_{k=0}^{\infty }\frac{t^{k}}{\left[ k\right] _{q}!}R_{k}\left(
x|q^{2},q\right) \allowbreak =\allowbreak \frac{1}{%
(1-(1-q)tx+(1-q)t^{2})(1-(1-q)txq+(1-q)t^{2}q^{2})}.
\end{equation*}
\newline
Hence $R_{k}\left( x|q^{2},q\right) /\left( q\right) _{k}\allowbreak
=\allowbreak \sum_{k=0}^{n}q^{k}U_{k}\left( x\sqrt{1-q}/2\right)
U_{n-k}\left( x\sqrt{1-q}/2\right) $ using common knowledge on the
properties of the generating functions. Simple 'generating functions'
argument shows that $\sum_{k=0}^{n}q^{k}U_{k}\left( x\sqrt{1-q}/2\right)
U_{n-k}\left( x\sqrt{1-q}/2\right) $ simplifies to $\sum_{j=0}^{\left\lfloor
n/2\right\rfloor }q^{j}\left[ n+1-2j\right] _{q}U_{n-2j}\left( x\sqrt{1-q}%
/2\right) .$ On the other hand since these polynomials are proportional to $%
R_{k}\left( x|q^{2},q\right) $ we know their 3-term recurrence and the
density that makes them orthogonal. Similarly for other cases $R_{k}\left(
x|q^{m},q\right) ,$ $m\geq 3.$ Besides notice that $\forall n\geq -1,x\in 
\mathbb{R}$ $\lim_{m\rightarrow \infty }R_{k}\left( x|q^{m},q\right)
\allowbreak =\allowbreak H_{n}\left( x|q\right) .$
\end{remark}

We will need also two families of auxiliary polynomials.

\subsection{Rogers-Szeg\"{o}}

These polynomials are defined by the equality: 
\begin{equation*}
s_{n}\left( x|q\right) \allowbreak =\allowbreak \sum_{k=0}^{n}\QATOPD[ ] {n}{%
k}_{q}x^{k},
\end{equation*}%
for $n\geq 0$ and $s_{-1}\left( x|q\right) \allowbreak =\allowbreak 0.$ They
will be playing here an auxiliary r\^{o}le. In particular one shows (see
e.g. \cite{IA}(13.1.7)) that: 
\begin{equation}
h_{n}\left( x|q\right) \allowbreak =\allowbreak e^{in\theta }s_{n}\left(
e^{-2i\theta }|q\right) ,  \label{cH}
\end{equation}%
where $x\allowbreak =\allowbreak \cos \theta ,$ and that: 
\begin{equation*}
\sup_{\left\vert x\right\vert \leq 1}\left\vert h_{n}\left( x|q\right)
\right\vert \leq s_{n}\left( 1|q\right) .
\end{equation*}%
In the sequel the following identities discovered by Carlitz (see Exercise
12.2(b) and 12.2(c) of \cite{IA}), true for $\left\vert q\right\vert
,\left\vert t\right\vert <1$ : 
\begin{equation}
\sum_{k=0}^{\infty }\frac{s_{k}\left( 1|q\right) t^{k}}{\left( q\right) _{k}}%
\allowbreak =\allowbreak \frac{1}{\left( t\right) _{\infty }^{2}}%
,\sum_{k=0}^{\infty }\frac{s_{k}^{2}\left( 1|q\right) t^{k}}{\left( q\right)
_{k}}\allowbreak =\allowbreak \frac{\left( t^{2}\right) _{\infty }}{\left(
t\right) _{\infty }^{4}},  \label{Car_id}
\end{equation}%
will allow to show convergence of many considered in the sequel series.

\subsection{$q^{-1}-$Hermite}

We will need also polynomials $\left\{ B_{n}\left( x|q\right) \right\}
_{n\geq -1}$ defined by the following 3-term recurrence:%
\begin{equation}
B_{n+1}\left( y|q\right) \allowbreak =\allowbreak -q^{n}yB_{n}\left(
y|q\right) +q^{n-1}\left[ n\right] _{q}B_{n-1}\left( y|q\right) ,  \label{Be}
\end{equation}%
for all $n\geq 0$ and with $B_{-1}\left( y|q\right) =0,$ $B_{0}\left(
y|q\right) =1$. One easily shows that $B_{n}\left( x|1\right) \allowbreak
=\allowbreak i^{n}H_{n}\left( ix\right) $ (compare \cite{bms}). We will also
sometimes need the 'continuous' or 'lower case version' of these polynomials
namely $b_{n}\left( y|q\right) \allowbreak =\allowbreak
(1-q)^{n/2}B_{n}\left( 2y/\sqrt{1-q}|q\right) $. It is easy to notice that
the polynomials $b_{n}$ satisfy the following $3-$term recurrence :%
\begin{equation}
b_{n+1}\left( y|q\right) =-2q^{n}yb_{n}\left( y|q\right)
+q^{n-1}(1-q^{n})b_{n-1}\left( y|q\right) ,  \label{maleB}
\end{equation}%
with $b_{-1}\left( y|q\right) \allowbreak =\allowbreak 0,$ $b_{0}\left(
y|q\right) \allowbreak =\allowbreak 1$. Moreover if we consider $\tilde{b}%
_{n}\left( y|q\right) \allowbreak =\allowbreak q^{-n(n-1)/2}i^{n}b_{n}\left(
iy|q\right) $ then we see that the polynomials $\tilde{b}_{n}$ satisfy the
following 3-term recurrence: 
\begin{equation*}
\tilde{b}_{n+1}\left( y|q\right) =2y\tilde{b}_{n}\left( y|q\right)
-(q^{-n}-1)\tilde{b}_{n-1}\left( y|q\right) ,
\end{equation*}%
hence $\tilde{b}_{n}^{\prime }s$ are orthogonal for $q>1.$ The point is that
there does not exit the unique measure that makes these polynomials
orthogonal. Discussion of this case is thoroughly done in \cite{IsMas94}.
However polynomials $\left\{ B_{n}\right\} $ will be of great help in the
sequel.

Notice (\cite{bms}) that 
\begin{equation*}
\sum_{k=0}^{\infty }\frac{t^{k}}{\left[ k\right] _{q}!}B_{k}\left(
x|q\right) \allowbreak =\allowbreak \prod_{j=0}^{\infty }V_{q}\left(
x|q^{j}t\right) .
\end{equation*}

Hence in particular we get: $B_{n}\left( y|0\right) \allowbreak =\allowbreak
\left\{ 
\begin{array}{ccc}
-y & if & n=1 \\ 
1 & if & n=2\vee n=0 \\ 
0 & if & n\geq 3%
\end{array}%
\right. .$ Comparing the above mentioned formulae we see that%
\begin{equation*}
\sum_{k=0}^{\infty }\frac{t^{k}}{\left[ k\right] _{q}!}B_{k}\left(
x|q\right) =1/\sum_{k=0}^{\infty }\frac{t^{k}}{\left[ k\right] _{q}!}%
H_{k}\left( x|q\right) .
\end{equation*}

\section{Connection coefficients and other useful finite expansions}

\subsection{Connection coefficients}

We consider $n\geq 0$

T\&U

\begin{eqnarray*}
T_{n}\left( x\right) \allowbreak &=&\allowbreak \left( U_{n}\left( x\right)
-U_{n-2}\left( x\right) \right) /2, \\
U_{n}\left( x\right) &=&2\sum_{k=0}^{\left\lfloor n/2\right\rfloor
}T_{n-2k}\left( x\right) \allowbreak -\allowbreak \left( 1+\left( -1\right)
^{n}\right) /2.
\end{eqnarray*}

These expansions belong to common knowledge of the special functions theory

H\&T

\begin{equation*}
H_{n}\left( x|q\right) \allowbreak =\allowbreak (1-q)^{-n/2}\sum_{k=0}^{n}%
\QATOPD[ ] {n}{k}_{q}T_{n-2k}\left( x\sqrt{1-q}/2\right) ,
\end{equation*}%
if one sets $T_{-n}\left( x\right) \allowbreak =\allowbreak T_{n}\left(
x\right) ,$ $n\geq 0.$

First notice that (\ref{cH}) is equivalent to $h_{n}\left( x\right)
\allowbreak =\allowbreak \sum_{k=0}^{n}\QATOPD[ ] {n}{k}_{q}\cos \left(
2k-n\right) \theta $ where $x=\cos \theta .$ Next we use (\ref{repCh})

H\&H

\begin{equation*}
H_{n}\left( x|p\right) \allowbreak =\allowbreak \sum_{k=0}^{\left\lfloor
n/2\right\rfloor }\tilde{C}_{n,n-2k}\left( p,q\right) H_{n-2k}\left(
x|q\right) ,
\end{equation*}%
where 
\begin{gather*}
\tilde{C}_{n,n-2k}\allowbreak (p,q)\allowbreak =\allowbreak \frac{\left(
1-q\right) ^{n/2-k}}{(1-p)^{n/2}}\times \\
\sum_{j=0}^{k}\left( -1\right) ^{j}p^{k-j}q^{j\left( j+1\right) /2}\QATOPD[ ]
{n-2k+j}{j}_{q}\allowbreak (\QATOPD[ ] {n}{k-j}_{p}-p^{n-2k+2j+1}\QATOPD[ ] {%
n}{k-j-1}_{p}).
\end{gather*}

This formula follows the 'change of base' formula for the continuous $q-$%
Hermite polynomials (i.e. polynomials $h_{n}$) in e.g. \cite{IS03}, \cite%
{bressoud} or \cite{GIS99} (formula 7.2) that states that: 
\begin{equation*}
h_{n}\left( x|p\right) =\sum_{k=0}^{\left\lfloor n/2\right\rfloor
}c_{n,n-2k}\left( p,q\right) h_{n-2k}\left( x|q\right) ,
\end{equation*}%
where 
\begin{equation*}
c_{n,n-2k}\left( p,q\right) =\frac{\left( 1-p\right) ^{n/2}}{\left(
1-q\right) ^{n/2-k}}\tilde{C}_{n,n-2k}\allowbreak (p,q)\allowbreak .
\end{equation*}

U\&H

\begin{gather*}
U_{n}\left( x\sqrt{1-q}/2\right) =\sum_{j=0}^{\left\lfloor n/2\right\rfloor
}\left( -1\right) ^{j}(1-q)^{n/2-j}q^{j\left( j+1\right) /2}\QATOPD[ ] {n-j}{%
j}_{q}H_{n-2j}\left( x|q\right) , \\
H_{n}\left( y|q\right) =\sum_{k=0}^{\left\lfloor n/2\right\rfloor
}(1-q)^{-n/2}\frac{q^{k}-q^{n-k+1}}{1-q^{n-k+1}}\QATOPD[ ] {n}{k}%
_{q}U_{n-2k}\left( y\sqrt{1-q}/2\right) .
\end{gather*}

These expansion follow the previous one setting once $p\allowbreak
=\allowbreak 0$ and then secondly $q=0$ and then $p=q.$

H\&bH

\begin{eqnarray}
h_{n}\left( x|a,q\right) \allowbreak &=&\allowbreak \sum_{k=0}^{n}\QATOPD[ ]
{n}{k}_{q}(-1)^{k}q^{\binom{k}{2}}a^{k}h_{n-k}\left( x|q\right) ,
\label{bigh} \\
H_{n}\left( x|a,q\right) \allowbreak &=&\sum_{k=0}^{n}\QATOPD[ ] {n}{k}%
_{q}(-1)^{k}q^{\binom{k}{2}}a^{k}H_{n-k}\left( x|q\right) .  \label{bigH}
\end{eqnarray}

(\ref{bigh}) it is formula 19 of \cite{Chen08} (see also \cite{Floreanini97}%
). (\ref{bigH}) is a simple consequence of (\ref{bigh}).

H\&P

\begin{eqnarray}
P_{n}\left( x|y,\rho ,q\right) &=&\sum_{j=0}^{n}\QATOPD[ ] {n}{j}\rho
^{n-j}B_{n-j}\left( y|q\right) H_{j}\left( x|q\right) ,  \label{PnaH} \\
H_{n}\left( x|q\right) &=&\sum_{j=0}^{n}\QATOPD[ ] {n}{j}\rho
^{n-j}H_{n-j}\left( y|q\right) P_{j}\left( x|y,\rho ,q\right) .  \label{HnaP}
\end{eqnarray}

For the proof of (\ref{PnaH}) see Remark 1 following Theorem 1 in \cite{bms}%
. For the proof of (\ref{HnaP}) we start with formula (4.7) in \cite{IRS99}
that gives connection coefficients of $h_{n}$ with respect to $Q_{n}.$ Then
we pass to the polynomials $H_{n}$ \& $P_{n}$ using formulae $h_{n}\left(
x|q\right) \allowbreak =\allowbreak \left( 1-q\right) ^{n/2}H_{n}\left( 
\frac{2x}{\sqrt{1-q}}|q\right) ,$ $n\geq 1$ and $p_{n}(x|a,b,q)\allowbreak
=\allowbreak \left( 1-q\right) ^{n/2}P_{n}\left( \frac{2x}{\sqrt{1-q}}|\frac{%
2a}{\sqrt{\left( 1-q\right) b}},\sqrt{b},q\right) .$ By the way notice that
this formula can be easily derived from assertions iv) and (\ref{suma BH})
with $m\allowbreak =\allowbreak 0$ presented below and the standard change
of order of summation. Now it remains to return to polynomials $H_{n}$ .

As a corollary of (\ref{HnaP}) and (\ref{PnPm}) we get a nice formula given
in \cite{bms}: For $\forall n\geq 1,\allowbreak \left\vert \rho \right\vert
<1,$\allowbreak $y\in S\left( q\right) $%
\begin{equation*}
\int_{S\left( q\right) }H_{n}\left( x|q\right) f_{CN}\left( x|y,\rho
,q\right) dx=\rho ^{n}H_{n}\left( y|q\right) .
\end{equation*}

bH\&P

\begin{eqnarray}
H_{n}\left( x|a,q\right) \allowbreak &=&\allowbreak \sum_{j=0}^{n}\QATOPD[ ]
{n}{j}_{q}P_{j}\left( x|y,\frac{a}{b},q\right) \left( \frac{a}{b}\right)
^{n-j}H_{n-j}\left( y|b,q\right) ,  \label{bHna P} \\
P_{n}\left( x|y,\rho ,q\right) \allowbreak &=&\allowbreak \sum_{k=0}^{n}%
\QATOPD[ ] {n}{k}_{q}\rho ^{n-k}B_{n-k}\left( x|a/\rho ,q\right) H_{k}\left(
x|a,q\right) ,  \label{PnabH}
\end{eqnarray}

where we denoted $B_{m}\left( x|b,q\right) \overset{df}{=}\sum_{j=0}^{m}%
\QATOPD[ ] {m}{j}_{q}b^{m-j}B_{j}\left( x|q\right) .$ Strict proof of (\ref%
{bHna P}) and (\ref{PnabH}) is presented in \cite{SzablKer}. It is easy and
is based on (\ref{PnaH}) and (\ref{bigH}).

P\&P

\begin{eqnarray}
P_{n}\left( x|y,\rho ,q\right) &=&\sum_{j=0}^{n}\QATOPD[ ] {n}{j}%
_{q}r^{n-j}P_{j}\left( x|z,r,q\right) P_{n-j}(z|y,\rho /r,q),  \label{PnaP}
\\
\frac{P_{n}\left( y|z,t,q\right) }{(t^{2})_{n}}\allowbreak &=&\allowbreak
\sum_{j=0}^{n}(-1)^{j}q^{j(j-1)/2}\QATOPD[ ] {n}{j}_{q}t^{j}H_{n-j}\left(
y|q\right) \frac{P_{j}\left( z|y,t,q\right) }{\left( t^{2}\right) _{j}},
\label{odwrocenie}
\end{eqnarray}%
if one extends definition of polynomials $P_{n}$ for $\left\vert \rho
\right\vert >1$ by (\ref{PnaH}). (\ref{PnaP}) has been proved in \cite%
{SzablKer}, while (\ref{odwrocenie}) is given in \cite{Szab6} Corollary 2.
Besides it follows directly from one of the infinite expansions that will be
presented in section \ref{nieskon}.

As a corollary of (\ref{odwrocenie}) and of course (\ref{PnPm}) we get the
following formula: \newline
For $\forall n\geq 1,\allowbreak \left\vert \rho \right\vert <1,$\allowbreak 
$x\in S\left( q\right) $%
\begin{equation*}
\int_{S\left( q\right) }P_{n}\left( x|y,\rho ,q\right) f_{CN}\left( y|x,\rho
,q\right) dy=\left( \rho ^{2}\right) _{n}H_{n}\left( x|q\right) .
\end{equation*}

R\&R

For $\left\vert \beta \right\vert ,\left\vert \gamma \right\vert <1:$%
\begin{equation}
R_{n}\left( x|\gamma ,q\right) \allowbreak =\allowbreak
\sum_{k=0}^{\left\lfloor n/2\right\rfloor }\frac{\left[ n\right] _{q}!\beta
^{k}\left( \gamma /\beta \right) _{k}\left( \gamma \right) _{n-k}\left(
1-\beta q^{n-2k}\right) }{\left[ k\right] _{q}!\left[ n-2k\right]
_{q}!\left( \beta q\right) _{n-k}\left( 1-\beta \right) }R_{n-2k}\left(
x|\beta ,q\right) .  \label{RnaR}
\end{equation}%
(\ref{RnaR}) is in fact celebrated connection coefficient formula for the
Rogers polynomials which was in fact expressed in terms of polynomials $%
C_{n}.$ For details see \cite{IA},(13.3.1).

R\&H

For $\left\vert \beta \right\vert ,\left\vert \gamma \right\vert <1:$ 
\begin{eqnarray}
R_{n}\left( x|\gamma ,q\right) \allowbreak &=&\allowbreak
\sum_{k=0}^{\left\lfloor n/2\right\rfloor }(-1)^{k}\frac{q^{k(k-1)/2}\left[ n%
\right] _{q}!\gamma ^{k}\left( \gamma \right) _{n-k}}{\left[ k\right] _{q}!%
\left[ n-2k\right] _{q}!}H_{n-2k}\left( x|q\right) ,  \label{RnaH} \\
H_{n}\left( x|q\right) \allowbreak &=&\allowbreak \sum_{k=0}^{\left\lfloor
n/2\right\rfloor }\frac{\left[ n\right] _{q}!}{\left[ k\right] _{q}!\left[
n-2k\right] _{q}!}\frac{\beta ^{k}\left( 1-\beta q^{n-2k}\right) }{(1-\beta
)\left( \beta q\right) _{n-k}}R_{n-2k}\left( x|\beta ,q\right) .
\label{HnaR}
\end{eqnarray}

(\ref{RnaH}) and (\ref{HnaR}) are particular cases of (\ref{RnaR}), first
for $\beta \allowbreak =\allowbreak 0$ and the second for $\gamma
\allowbreak =\allowbreak 0.$

B\&H

\begin{equation}
B_{n}\left( x|q\right) \allowbreak =\allowbreak \left( -1\right) ^{n}q^{%
\binom{n}{2}}\sum_{k=0}^{\left\lfloor n/2\right\rfloor }\QATOPD[ ] {n}{k}_{q}%
\QATOPD[ ] {n-k}{k}_{q}\left[ k\right] _{q}!q^{k(k-n)}H_{n-2k}\left(
x|q\right) .  \label{BnaH}
\end{equation}

(\ref{BnaH}) was proved in \cite{Szab6} Lemma 2 assertion i).

As an immediate observation we have the following expansion of the ASC
polynomials in the $q$-Hermite polynomials.

\begin{proposition}
\begin{gather*}
P_{n}\left( x|y,\rho ,q\right) \allowbreak =\allowbreak
\sum_{k=0}^{\left\lfloor n/2\right\rfloor }\QATOPD[ ] {n}{k}_{q}\QATOPD[ ] {%
n-k}{k}_{q}\left[ k\right] _{q}!q^{k(k-1)}\rho ^{2k}\times \\
\sum_{s=0}^{n-2k}(-1)^{s}\QATOPD[ ] {n-2k}{s}_{q}q^{\binom{s}{2}}(q^{k}\rho
)^{s}H_{n-2k-s}\left( x|q\right) H_{s}\left( y|q\right) .
\end{gather*}
\end{proposition}

\begin{proof}
First we use (\ref{PnaH}) and then (\ref{BnaH}) obtaining: $P_{n}\left(
x|y,\rho ,q\right) \allowbreak =\allowbreak \sum_{s=0}^{n}\QATOPD[ ] {n}{s}%
_{q}H_{n-s}\left( x|q\right) \rho ^{s}(-1)^{s}q^{\binom{s}{2}}\allowbreak
\times \allowbreak \sum_{k=0}^{\left\lfloor s/2\right\rfloor }\QATOPD[ ] {s}{%
k}_{q}\QATOPD[ ] {n-k}{k}_{q}\left[ k\right] _{q}!q^{k(k-s)}H_{s-2k}\left(
y|q\right) $ . Now we change the order of summation.
\end{proof}

\subsection{Useful finite expansions}

We start with the so called 'linearization formulae'. These are the formulae
expressing a product of two or more polynomials of the same type as linear
combinations of polynomials of the same type as the ones produced. We will
extend the name 'linearization formulae' by relaxing the requirement of
polynomials involved to be of the same type. Generally to obtain
'linearization formula ' is not simple and requires a lot of tedious
calculations.

\subsubsection{Linearization formulae}

$q-$Hermite polynomials

The formulae below can be found in e.g. \cite{IA} (Thm. 13.1.5) and also in 
\cite{ALIs88} and originally were formulated for polynomials $h_{n}.$ Here
below are presented for polynomials $H_{n}$ using (\ref{ch_of_vars}):

\begin{gather}
H_{n}\left( x|q\right) H_{m}\left( x|q\right) =\sum_{j=0}^{\min \left(
n,m\right) }\QATOPD[ ] {m}{j}_{q}\QATOPD[ ] {n}{j}_{q}\left[ j\right]
_{q}!H_{n+m-2j}\left( x|q\right) ,  \label{identity} \\
H_{n}\left( x|q\right) H_{m}\left( x|q\right) H_{k}\left( x|q\right) =
\label{identyty31} \\
\sum_{r,s}\QATOPD[ ] {m}{r}_{q}\QATOPD[ ] {n}{r}_{q}\QATOPD[ ] {k}{s}_{q}%
\QATOPD[ ] {m+n-2r}{s}_{q}\left[ s\right] _{q}!\left[ r\right]
_{q}!H_{n+m+k-2r-2s}\left( x|q\right) =  \label{identity32} \\
\sum_{j=0}^{\left\lfloor (k+m+n)/2\right\rfloor }\left( \sum_{r=\max
(j-k,0)}^{\min (m,n,m+n-j)}\QATOPD[ ] {m}{r}_{q}\QATOPD[ ] {n}{r}_{q}\QATOPD[
] {k}{j-r}_{q}\QATOPD[ ] {m+n-2r}{j-r}_{q}\left[ r\right] _{q}\left[ j-r%
\right] _{q}\right) H_{n+m+k-2j}\left( x|q\right) .  \label{identyty33}
\end{gather}

In fact (\ref{identity}) can be easily derived (by re-scaling and changing
of variables) from an old result of Carlitz (\cite{Carlitz57}) that was
formulated in terms of the Rogers-Szeg\"{o} $\left\{ s_{n}\left( x|q\right)
\right\} _{n\geq -1}$ polynomials. Carlitz proved in the same paper another
useful identity concerning polynomials $s_{n}$ that can be easily
reformulated in terms of the polynomials $H_{n}.$ The formula below is in a
sense an inverse of (\ref{identity}). Namely we have:

\begin{equation}
H_{n+m}\left( x\right) =\sum_{k=0}^{\min \left( n,m\right) }(-1)^{k}q^{%
\binom{k}{2}}\QATOPD[ ] {m}{k}_{q}\QATOPD[ ] {n}{k}_{q}\left[ k\right]
_{q}!H_{n-k}\left( x\right) H_{m-k}\left( x\right) .  \label{identyty2}
\end{equation}

H\&B

$\forall n,m\geq 1:$%
\begin{equation}
H_{m}\left( x|q\right) B_{n}\left( x|q\right) =(-1)^{n}q^{\binom{n}{2}%
}\sum_{k=0}^{\left\lfloor (n+m)/2\right\rfloor }\QATOPD[ ] {n}{k}_{q}\QATOPD[
] {n+m-k}{k}_{q}\left[ k\right] _{q}!q^{-k(n-k)}H_{n+m-2k}\left( x|q\right) .
\label{BHnaH}
\end{equation}

This formula having technical importance has been proved in \cite{Szab6}
Lemma 2 assertion ii).

H\&R

We have also useful formula:

$\forall n,m\geq 1:$%
\begin{equation}
H_{m}\left( x|q\right) R_{n}\left( x|\beta ,q\right) =\sum_{k,j}\QATOPD[ ] {m%
}{j}_{q}\QATOPD[ ] {n}{k+j}_{q}\QATOPD[ ] {n-k-j}{k}\left[ k+j\right]
_{q}!\left( -\beta \right) ^{k}q^{\binom{k}{2}}\left( \beta \right)
_{n-k}H_{n+m-2k-2j}\left( x|q\right) ,  \label{HRnaH}
\end{equation}

Which was proved in \cite{ALIs88} (1.9) for $h_{n}$ and $C_{n}$ and then
modified using (\ref{ch_of_vars}) and (\ref{q-US}).

Q\&Q

For completeness let us mention that in \cite{Stanton08} there is given a
very complicated linearization formula for Al-Salam--Chihara polynomials
given in Theorem 1.

\subsubsection{Useful finite sums and identities}

We have also the following a very useful generalization of formula (1.12) of 
\cite{bms} which was proved in \cite{Szab6} (Lemma2 assertion i)).

For all $n\geq 0:$%
\begin{equation}
\sum_{k=0}^{n}\QATOPD[ ] {n}{k}_{q}B_{n-k}\left( x|q\right) H_{k+m}\left(
x|q\right) \allowbreak =\left\{ 
\begin{array}{ccc}
0 & if & n>m \\ 
(-1)^{n}q^{\binom{n}{2}}\frac{\left[ m\right] _{q}!}{\left[ m-n\right] _{q}!}%
H_{m-n}\left( x|q\right) & if & m\geq n%
\end{array}%
\right. .  \label{suma BH}
\end{equation}

Let us remark that for $q\allowbreak =\allowbreak 0$ (\ref{suma BH}) reduces
to $3$-term recurrence of polynomials $U_{n}\left( x/2\right) .$

For $q\allowbreak =\allowbreak 1$ we get 
\begin{equation*}
\sum_{k=0}^{n}\binom{n}{k}i^{n-k}H_{n-k}\left( ix\right) H_{k+m}\left(
x\right) \allowbreak =\allowbreak \left\{ 
\begin{array}{ccc}
0 & if & n>m \\ 
(-1)^{n}\frac{m!}{(m-n)!}H_{m-n}\left( x\right) & if & m\geq n%
\end{array}%
\right. .
\end{equation*}

Recently in \cite{Szab-bAW} the following identities involving ASC
polynomials $p_{n}$ were given:

i) $\forall n\geq 1,0\leq k<n,z,y,t\in \mathbb{R}:$%
\begin{equation*}
\sum_{j=0}^{n-k}\QATOPD[ ] {n-k}{j}_{q}\frac{p_{j}\left( z|y,tq^{k},q\right) 
}{\left( t^{2}q^{2k}\right) _{j}}\frac{g_{n-k-j}\left( z|y,tq^{n-1},q\right) 
}{\left( t^{2}q^{n+j+k-1}\right) _{n-k-j}}\allowbreak =\allowbreak 0,
\end{equation*}

ii) $\forall n\geq 1,0\leq k<n,z,y,t\in \mathbb{R}:$%
\begin{equation*}
\sum_{m=0}^{n-k}\QATOPD[ ] {n-k}{m}_{q}\frac{p_{n-k-m}\left(
z|y,tq^{m+k},q\right) g_{m}(z|y,tq^{m+k-1},q)}{\left( t^{2}q^{2m+2k}\right)
_{n-k-m}\left( t^{2}q^{m+2k-1}\right) _{m}}\allowbreak =\allowbreak 0,
\end{equation*}%
where polynomials $g_{n}$ are somewhat analogous to polynomials $b_{n}$ and
are defined by the formula:%
\begin{equation}
g_{n}\left( x|y,\rho ,q\right) \allowbreak =\allowbreak \left\{ 
\begin{array}{ccc}
\rho ^{n}p_{n}\left( y|x,\rho ^{-1},q\right) & if & \rho \neq 0 \\ 
b_{n}\left( x|q\right) & if & \rho =0%
\end{array}%
\right. .  \label{_g}
\end{equation}%
Similar ones involving polynomials $P_{n}$ and appropriately modified
polynomials $g_{n}$ were also presented in \cite{Szab-bAW}.

Let us mention that polynomials $g_{n}$ play with respect to polynomials $%
p_{n}$ similar r\^{o}le as polynomials $b_{n}$ with respect to polynomials $%
h_{n}.$ Namely we have:

for all $\left\vert t\right\vert ,\left\vert q\right\vert ,\left\vert \rho
\right\vert <1,~\left\vert x\right\vert ,\left\vert y\right\vert \leq 1:$%
\begin{equation*}
\sum_{j=0}^{\infty }\frac{t^{n}}{\left( q\right) _{n}}g_{n}\left( x|y,\rho
,q\right) \allowbreak =\allowbreak 1/\varphi _{p}\left( x,t|y,\rho ,q\right)
,
\end{equation*}

for all $n\geq 1,x,y,\rho \in \mathbb{R}:$%
\begin{equation*}
\sum_{j=0}^{n}\QATOPD[ ] {n}{j}_{q}p_{j}\left( x|y,\rho ,q\right)
g_{n-j}\left( x|y,\rho ,q\right) =0.
\end{equation*}

\section{Infinite expansions\label{nieskon}}

\subsection{Kernels}

We start with the famous Poisson--Mehler expansion of $f_{CN}\left( x|y,\rho
,q\right) /f_{N}\left( x|q\right) $ in the an infinite series of Mercier's
type (compare e.g. \cite{Mercier09}). Namely the following fact is true:

\begin{theorem}
\label{Mehler}$\forall \left\vert q\right\vert ,\left\vert \rho \right\vert
<1;x,y\in S\left( q\right) :$%
\begin{gather}
\frac{(\rho ^{2})_{\infty }}{\prod_{k=0}^{\infty }W_{q}\left( x,y|\rho
q^{k}\right) }  \label{PM} \\
=\sum_{n=0}^{\infty }\frac{\rho ^{n}}{[n]_{q}!}H_{n}(x|q)H_{n}(y|q).  \notag
\end{gather}

For $q\allowbreak =\allowbreak 1,$ $x,y\in \mathbb{R}$ we have%
\begin{equation}
\frac{\exp \left( \frac{x^{2}+y^{2}}{2}\right) }{\sqrt{1-\rho ^{2}}}\exp (-%
\frac{x^{2}+y^{2}-2\rho xy}{2(1-\rho ^{2})})\allowbreak =\allowbreak
\sum_{n=0}^{\infty }\frac{\rho ^{n}}{n!}H_{n}(x)H_{n}(y).  \label{PMq=1}
\end{equation}
\end{theorem}

\begin{proof}
There exist many proofs of both formulae (see e.g. \cite{IA}, \cite{bressoud}%
). One of the shortest, exploiting connection coefficients, given in (\ref%
{PnaH}) is given in \cite{Szab4}.
\end{proof}

\begin{corollary}
$\forall \left\vert q\right\vert ,\left\vert \rho \right\vert <1;x\in
S\left( q\right) :$%
\begin{equation*}
\sum_{k\geq 0}\frac{\rho ^{k}\left( \rho q^{k-1}\right) _{\infty }}{\left[ k%
\right] _{q}!}H_{2k}\left( x|q\right) \allowbreak =\allowbreak \frac{\left(
\rho ^{2}\right) _{\infty }}{\left( \rho \right) _{\infty }}%
\prod_{k=0}^{\infty }L_{q}^{-1}\left( x|\rho q^{k}\right) .
\end{equation*}
\end{corollary}

\begin{proof}
We put $y\allowbreak =\allowbreak x$ in (\ref{PM}), then we apply (\ref%
{identity}), change order of summation and finally apply formulae $\frac{1}{%
\left( \rho \right) _{j+1}}\allowbreak =\allowbreak \sum_{k\geq 0}\QATOPD[ ]
{j+k}{k}_{q}\rho ^{k}$ and $\frac{\left( \rho \right) _{\infty }}{\left(
\rho \right) _{j+1}}\allowbreak =\allowbreak \left( q^{j-1}\rho \right)
_{\infty }$
\end{proof}

We will call expression of the form of the right hand side of (\ref{PM}) the
kernel expansion while the expressions from the left hand side of (\ref{PM})
kernels. The name refers to Mercier's theorem and the fact that for example 
\begin{equation*}
\int_{S\left( q\right) }k\left( x,y|\rho ,q\right) H_{n}\left( x|q\right)
f_{N}\left( x|q\right) dx=\rho ^{n}H_{n}\left( y|q\right) f_{N}\left(
y|q\right) ,
\end{equation*}%
where we denoted by $k\left( x,y|\rho ,q\right) $ the left hand side of (\ref%
{PM}). Hence we see that $k$ is a kernel, while function $H_{n}\left(
x|q\right) f_{N}\left( x|q\right) $ are eigenfunctions of kernel $k$ with $%
\rho ^{n}$ being an eigenvalue related to an eigenfunction $H_{n}\left(
x|q\right) f_{N}\left( x|q\right) .$ Such kernels and kernel expansions are
very important in analysis or quantum physics in the analysis of different
models of harmonic oscillators.

In the literature however there is small confusion concerning terminology.
Sometimes expression of the form $\sum_{n\geq 0}a_{n}p_{n}\left( x\right)
p_{n}\left( y\right) $ where $\left\{ p_{n}\right\} $ is is a family of
polynomials are also called kernels (like in \cite{suslov96})) or even
sometimes 'bilinear generating function' (see e.g. \cite{Rahman97})) or also
Poisson kernels. If say $p_{n}\left( y\right) $ is substituted by say $%
q_{n}\left( y\right) $ then one says that we deal with the non-symmetric
kernel.

The process of expressing these sums in a closed form is then called
'summing of kernels'.

Summing the kernel expansions is difficult. Proving positivity of the
kernels is another difficult problem. Only some are known and have
relatively simple forms. In most cases sums are in the form of a complex
finite sum of the so called basic hypergeometric functions. Below we will
present several of them. Mostly the ones involving the big q-Hermite,
Al-Salam--Chihara and $q$-ultraspherical polynomials.

To present more complicated sums we will need the following definition of
the basic hypergeometric function namely 
\begin{equation}
_{j}\phi _{k}\left[ 
\begin{array}{cccc}
a_{1} & a_{2} & \ldots & a_{j} \\ 
b_{1} & b_{2} & \ldots & b_{k}%
\end{array}%
;q,x\right] =\sum_{n=0}^{\infty }\frac{\left( a_{1},\ldots ,a_{j}|q\right) }{%
\left( b_{1},\ldots ,b_{k}|q\right) }\left( \left( -1\right) ^{n}q^{\binom{n%
}{2}}\right) ^{1+k-j}x^{n},  \label{jfk}
\end{equation}

\begin{equation}
_{2m}W_{2m-1}\left( a,a_{1},\ldots ,a_{2m-3};q,x\right) =_{2m}\phi _{2m-1} 
\left[ 
\begin{array}{ccccccc}
a & q\sqrt{a} & -q\sqrt{a} & a_{1} & a_{2} & \ldots & a_{2m-3} \\ 
\sqrt{a} & -\sqrt{a} & \frac{qa}{a_{1}} & \frac{qa}{a_{2}} & \ldots & \frac{%
qa}{a_{2m-3}} & 
\end{array}%
;q,x\right] .  \label{2mW2m-1}
\end{equation}

We will now present the kernels built of families of polynomials that are
discussed here and their sums.

\begin{theorem}
\label{Kernels}i) For all $\left\vert t\right\vert <1,\left\vert
x\right\vert ,\left\vert y\right\vert <2:$ 
\begin{equation*}
\sum_{n=0}^{\infty }t^{n}U_{n}\left( x/2\right) U_{n}\left( y/2\right)
\allowbreak =\allowbreak \frac{\left( 1-t^{2}\right) }{\left( \left(
1-t^{2}\right) ^{2}-t\left( 1+t^{2}\right) xy+t^{2}(x^{2}+y^{2})\right) }.
\end{equation*}

ii) For all $\left\vert t\right\vert <1,\left\vert x\right\vert ,\left\vert
y\right\vert <1:$%
\begin{gather*}
\sum_{n=0}^{\infty }\frac{\left( 1-\beta q^{n}\right) \left( q\right) _{n}}{%
\left( 1-\beta \right) \left( \beta ^{2}\right) _{n}}t^{n}C_{n}\left(
x|\beta ,q\right) C_{n}\left( y|\beta ,q\right) \allowbreak = \\
\frac{\left( \beta q\right) _{\infty }^{2}}{\left( \beta ^{2}\right)
_{\infty }\left( \beta t^{2}\right) _{\infty }}\prod_{n=0}^{\infty }\frac{%
w\left( x,y|t\beta q^{n}\right) }{w\left( x,y|tq^{n}\right) }\times \\
~_{8}W_{7}\left( \frac{\beta t^{2}}{q},\frac{\beta }{q},te^{i(\theta +\phi
)},te^{-i\left( \theta +\phi \right) },te^{i\left( \theta -\phi \right)
},te^{-i\left( \theta -\phi \right) };q,\beta q\right) ,
\end{gather*}%
$\allowbreak $\newline
where $x\allowbreak =\allowbreak \cos \theta ,$ $y\allowbreak =\allowbreak
\cos \phi .$

iii) For all $\left\vert x\right\vert ,\left\vert y\right\vert ,\left\vert
t\right\vert ,\left\vert tb/a\right\vert \leq 1$ :%
\begin{gather}
\sum_{n\geq 0}\frac{\left( tb/a\right) ^{n}}{\left( q\right) _{n}}%
h_{n}\left( x|a,q\right) h_{n}\left( y|b,q\right) \allowbreak =\allowbreak
\left( \frac{b^{2}t^{2}}{a^{2}}\right) _{\infty }\prod_{k=0}^{\infty }\frac{%
v\left( x|tbq^{k}\right) }{w\left( x,y|t\frac{b}{a}q^{k}\right) }\times
\label{kernel_bigH} \\
_{3}\phi _{2}\left( 
\begin{array}{ccc}
t & bte^{i\left( \theta +\phi \right) }/a & bte^{i\left( -\theta +\phi
\right) }/a \\ 
b^{2}t^{2}/a^{2} & bte^{i\phi } & 
\end{array}%
;q,be^{-i\phi }\right) ,  \notag
\end{gather}%
with $x\allowbreak =\allowbreak \cos \theta $ and $y\allowbreak =\allowbreak
\cos \phi .$

iv) For all $\left\vert t\right\vert <1,x,y\in S\left( q\right) ,ab=\alpha
\beta :$ 
\begin{gather*}
\sum_{n\geq 0}\frac{\left( t\alpha /a\right) ^{n}}{\left( q\right)
_{n}\left( ab\right) _{n}}Q_{n}\left( x|a,b,q\right) Q_{n}\left( y|\alpha
,\beta ,q\right) \allowbreak =\allowbreak \\
\frac{\left( \frac{\alpha ^{2}t^{2}}{a},\frac{\alpha ^{2}t}{a}e^{i\theta
},be^{-i\theta },bte^{i\theta },\alpha te^{-i\phi },\alpha te^{i\phi
}\right) _{\infty }}{\left( ab,\frac{\alpha ^{2}t^{2}}{a}e^{i\theta }\right)
_{\infty }\prod_{k=0}^{\infty }w\left( x,y|\frac{\alpha t}{a}q^{k}\right) }%
~_{8}W_{7}\left( \frac{\alpha ^{2}t^{2}e^{i\theta }}{aq},t,\frac{\alpha t}{%
\beta },ae^{i\theta },\frac{\alpha t}{a}e^{i\left( \theta +\phi \right) },%
\frac{\alpha t}{a}e^{i\left( \theta -\phi \right) };q,be^{-i\theta }\right) ,
\end{gather*}%
where as before $x\allowbreak =\allowbreak \cos \theta $ and $y\allowbreak
=\allowbreak \cos \phi $ and \newline
\begin{gather*}
\sum_{n\geq 0}\frac{t^{n}}{\left( q\right) _{n}\left( ab\right) _{n}}%
Q_{n}\left( x|a,b,q\right) Q_{n}\left( y|\alpha ,\beta ,q\right) \allowbreak
=\allowbreak \\
\frac{\left( \frac{\beta t}{a}\right) _{\infty }}{\left( \alpha at\right)
_{\infty }}\prod_{k=0}^{\infty }\frac{(1+\alpha ^{2}t^{2}q^{2k})^{2}-2\alpha
tq^{k}\left( x+y\right) \left( 1+\alpha ^{2}t^{2}q^{2k}\right) +4\alpha
^{2}xyt^{2}q^{2k}}{w\left( x,y|tq^{k}\right) } \\
~_{8}W_{7}\left( \frac{\alpha at}{q},\frac{\alpha t}{b},ae^{i\theta
},ae^{-i\theta },\alpha e^{i\phi },\alpha e^{-i\phi };q;\frac{\beta t}{a}%
\right) .
\end{gather*}%
\newline

v) For all $\left\vert \rho _{1}\right\vert ,\left\vert \rho _{2}\right\vert
,\left\vert q\right\vert <1,$ $x,y\in S\left( q\right) $ 
\begin{equation}
0\leq \sum_{n\geq 0}\frac{\rho _{1}^{n}}{\left[ n\right] _{q}!\left( \rho
_{2}^{2}\right) _{n}}P_{n}\left( x|y,\rho _{2},q\right) P_{n}\left( z|y,%
\frac{\rho _{2}}{\rho _{1}},q\right) =\frac{\left( \rho _{1}^{2}\right)
_{\infty }}{\left( \rho _{2}^{2}\right) _{\infty }}\prod_{k=0}^{\infty }%
\frac{W_{q}\left( x,z|\rho _{2}q^{k}\right) }{W_{q}\left( x,y|\rho
_{1}q^{k}\right) }.  \label{kerASC}
\end{equation}
\end{theorem}

\begin{proof}[Remarks concerning the proof]
i) We set $q\allowbreak =\allowbreak 0$ in (\ref{PM}) and use (\ref{q=0}).
ii) It is formula (1.7) in \cite{Rahman97} based on \cite{GasRah}. iii) it
is formula (14.14) in \cite{suslov96}, iv) these are formulae (14.5) and
(14.8) of \cite{suslov96}. v) Notice that it cannot be derived from
assertion iv) since the condition $ab\allowbreak =\allowbreak \alpha \beta $
is not satisfied. Recall that (see (\ref{podstawienie})) $ab\allowbreak
=\allowbreak \rho _{2}^{2}$ while $\alpha \beta \allowbreak =\allowbreak
\rho _{1}^{2}.$ For the proof recall the idea of expansion of ratio of
densities presented in \cite{Szab4}, use formulae (\ref{PnaP}) and (\ref%
{PnPm}) and finally notice that $f_{CN}\left( x|y,\rho _{1}q\right)
/f_{CN}\left( x|z,\rho _{2},q\right) \allowbreak =\allowbreak \frac{\left(
\rho _{1}^{2}\right) _{\infty }}{\left( \rho _{2}^{2}\right) _{\infty }}%
\prod_{k=0}^{\infty }\frac{W_{q}\left( x,z|\rho _{2}q^{k}\right) }{%
W_{q}\left( x,y|\rho _{1}q^{k}\right) }.$
\end{proof}

\begin{corollary}
For all $\left\vert a\right\vert >\left\vert b\right\vert ,$ $x,y\in S\left(
q\right) :$%
\begin{equation*}
0\leq \sum_{n\geq 0}\frac{b^{n}}{\left[ n\right] _{q}!a^{n}}H_{n}\left(
x|a,q\right) H_{n}\left( y|b,q\right) \allowbreak =\allowbreak \left( \frac{%
b^{2}}{a^{2}}\right) _{\infty }\prod_{k=0}^{\infty }\frac{V_{q}\left(
x|bq^{k}\right) }{W_{q}\left( x,y|\frac{b}{a}q^{k}\right) }.
\end{equation*}
\end{corollary}

\begin{proof}
We set $t\allowbreak =\allowbreak 0$ in (\ref{kernel_bigH}) and assume $%
\left\vert b\right\vert <\left\vert a\right\vert .$ For an alternative
simple proof see \cite{SzablKer}.
\end{proof}

\subsection{Other infinite expansions}

In this subsection we will present some expansions that can be viewed as
reciprocals of some presented above expansions and also some generalizations
of so called Kibble--Slepian formula.

We start with some reciprocals of the expansions obtained above.

\subsubsection{Expansions of kernel's reciprocals}

\begin{theorem}
i) For $\left\vert q\right\vert ,\left\vert \rho \right\vert <1,x,y\in
S\left( q\right) :$%
\begin{equation*}
1/\sum_{n=0}^{\infty }\frac{\rho ^{n}}{\left[ n\right] _{q}!}H_{n}\left(
x|q\right) H_{n}\left( y|q\right) \allowbreak =\allowbreak
\sum_{n=0}^{\infty }\frac{\rho ^{n}}{\left( \rho ^{2}\right) _{n}[n]_{q}!}%
B_{n}\left( y|q\right) P_{n}\left( x|y,\rho ,q\right) .
\end{equation*}

ii) For $x,y\in \mathbb{R}$ and $\rho ^{2}<1/2$%
\begin{equation*}
1/\sum_{n=0}^{\infty }\frac{\rho ^{n}}{n!}H_{n}\left( x\right) H_{n}\left(
y\right) \allowbreak =\allowbreak \sum_{n=0}^{\infty }\frac{\rho ^{n}i^{n}}{%
n!\left( 1-\rho ^{2}\right) ^{n/2}}H_{n}\left( ix\right) H_{n}\left( \frac{%
(x-\rho y)}{\sqrt{1-\rho ^{2}}}\right) .
\end{equation*}

iii) For $\left\vert q\right\vert <1,\left\vert a\right\vert <\left\vert
b\right\vert ,x,y\in S\left( q\right) :$ 
\begin{equation*}
1/\sum_{n\geq 0}\frac{a^{n}}{\left[ n\right] _{q}!b^{n}}H_{n}\left(
x|a,q\right) H_{n}\left( y|b,q\right) =\sum_{n\geq 0}\frac{a^{n}}{\left[ n%
\right] _{q}!b^{n}\left( a^{2}/b^{2}\right) _{n}}B_{n}\left( y|b,q\right)
P_{n}\left( x|y,a/b,q\right) .
\end{equation*}

iv) For $\left\vert \rho _{1}\right\vert ,\left\vert \rho _{2}\right\vert
,\left\vert q\right\vert <1,$ $x,y\in S\left( q\right) :$%
\begin{equation*}
1/\sum_{n\geq 0}\frac{\rho _{1}^{n}}{\left[ n\right] _{q}!\left( \rho
_{2}^{2}\right) _{n}}P_{n}\left( x|y,\rho _{2},q\right) P_{n}\left( z|y,%
\frac{\rho _{2}}{\rho _{1}},q\right) =\sum_{n\geq 0}\frac{\rho _{2}^{n}}{%
\left[ n\right] _{q}!\left( \rho _{1}^{2}\right) _{n}}P_{n}\left( x|z,\rho
_{1},q\right) P_{n}\left( y|z,\frac{\rho _{1}}{\rho _{2}},q\right) .
\end{equation*}
\end{theorem}

\begin{proof}[Remarks concerning the proof]
i) and ii) are proved in \cite{Szab4}. iii) is proved in \cite{SzablKer}.
iv) directly follows (\ref{kerASC})
\end{proof}

\subsubsection{Some auxiliary infinite expansions}

The result below can be viewed as summing certain non-symmetric kernel.

\begin{lemma}
\label{UogCarl}For $x,y\in S\left( q\right) ,$ $\left\vert \rho \right\vert
<1$ let us denote 
\begin{equation*}
\gamma _{m,k}\left( x,y|\rho ,q\right) \allowbreak =\allowbreak
\sum_{k=0}^{\infty }\frac{\rho ^{k}}{\left[ k\right] _{q}!}H_{k+m}\left(
x|q\right) H_{k+k}\left( y|q\right) \allowbreak .
\end{equation*}
Then 
\begin{equation}
\gamma _{m,k}\left( x,y|\rho ,q\right) \allowbreak =\allowbreak \gamma
_{0,0}\left( x,y|\rho ,q\right) \Xi _{m,k}\left( x,y|\rho ,q\right) ,
\label{gamma_m_k}
\end{equation}%
where $\Xi _{m,k}$ is a polynomial in $x$ and $y$ of order at most $m+k.$ 
\newline
Further denote $C_{n}\left( x,y|\rho _{1},\rho _{2},\rho _{3},q\right)
\allowbreak =\allowbreak \sum_{k=0}^{n}\QATOPD[ ] {n}{k}_{q}\rho
_{1}^{n-k}\rho _{2}^{k}\Xi _{n-k},_{k}\left( x,y|\rho _{3,}q\right) .$ 
\newline
Then we have in particular

i) $\Xi _{m,k}\left( x,y|\rho ,q\right) \allowbreak =\allowbreak \Xi
_{k,m}\left( y,x|\rho ,q\right) ,$ 
\begin{equation*}
\Xi _{m,k}\left( x,y|\rho ,q\right) \allowbreak =\allowbreak
\sum_{s=0}^{k}(-1)^{s}q^{\binom{s}{2}}\QATOPD[ ] {k}{s}_{q}\rho
^{s}H_{k-s}\left( y|q\right) P_{m+s}(x|y,\rho ,q)/(\rho ^{2})_{m+s},
\end{equation*}

ii) and 
\begin{equation}
C_{n}\left( x,y|\rho _{1},\rho _{2},\rho _{3},q\right) \allowbreak
=\allowbreak \sum_{s=0}^{n}\QATOPD[ ] {n}{s}_{q}H_{n-s}\left( y|q\right)
P_{s}\left( x|y,\rho _{3},q\right) \rho _{1}^{n-s}\rho _{2}^{s}\left( \rho
_{1}\rho _{3}/\rho _{2}\right) _{s}/\left( \rho _{3}^{2}\right) _{s}.
\label{_C}
\end{equation}
\end{lemma}

\begin{proof}
Proof that $\gamma _{m,k}\left( x,y|\rho ,q\right) \allowbreak /\allowbreak
\gamma _{0,0}\left( x,y|\rho ,q\right) $ is a polynomial can be deduced from 
\cite{Carlitz72} (formula 1.4) where the result was formulated for
Rogers-Szeg\"{o} polynomials. To get the ii) from this result of Carlitz
using (\ref{cH}) is not easy. For the alternative, simple although lengthy
proof of the general case and other assertions we refer the reader to \cite%
{Szab5} and \cite{Szab6}.
\end{proof}

Exploring Carlitz paper \cite{Carlitz72} and confronting it with above Lemma %
\ref{UogCarl} we arrive at the following conversion Lemma.

\begin{lemma}
$\forall n,m\geq 0,$ $\left\vert t\right\vert <1,$ $\theta \in (-\pi ,\pi ]:$
\begin{eqnarray}
&&\sum_{k=0}^{m}\sum_{l=0}^{n}\QATOPD[ ] {m}{k}_{q}\QATOPD[ ] {n}{l}_{q}%
\frac{\left( te^{i\left( -\theta +\eta \right) }\right) _{k}\left(
te^{i\left( \theta -\eta \right) }\right) _{l}\left( te^{-i\left( \theta
+\eta \right) }\right) _{k+l}}{\left( t^{2}\right) _{k+l}}e^{-i\left(
m-2k\right) \theta }e^{-i(n-2l)\eta }  \label{upr_Car} \\
&&=\sum_{j=0}^{n}(-1)^{j}q^{\binom{j}{2}}\QATOPD[ ] {n}{j}%
_{q}t^{j}h_{n-j}(y|q)p_{m+j}\left( x|y,t,q\right) /\left( t^{2}\right)
_{j+m},  \notag
\end{eqnarray}%
with $x\allowbreak =\allowbreak \cos \theta $ and $y\allowbreak =\allowbreak
\cos \eta .$
\end{lemma}

\begin{proof}
See \cite{Szab-bAW} Proposition 6.
\end{proof}

\subsubsection{Generalization of Kibble--Slepian formula}

Recall that Kibble in 1949 \cite{Kibble} and independently Slepian in 1972 
\cite{Slepian72} extended the Poisson--Mehler formula to higher dimensions,
expanding ratio of the standardized multidimensional Gaussian density
divided by the product of one dimensional marginal densities in the multiple
sum involving only constants (correlation coefficients) and the Hermite
polynomials. The formula in its generality can be found in \cite{IA} (4.7.2
p.107). Since we are going to generalize its $3$-dimensional version only
this version will be presented here.

Namely let us consider $3$ dimensional density $f_{3D}\left(
x_{1},x_{2},x_{3};\rho _{12},\rho _{13},\rho _{23}\right) $ of Normal random
vector $N\left( \left[ 
\begin{array}{c}
0 \\ 
0 \\ 
0%
\end{array}%
\right] ,\left[ 
\begin{array}{ccc}
1 & \rho _{12} & \rho _{13} \\ 
\rho _{12} & 1 & \rho _{23} \\ 
\rho _{13} & \rho _{23} & 1%
\end{array}%
\right] \right) .$ Of course we must assume that the parameters $\rho _{12},$
$\rho _{13},$ $\rho _{23}$ are such that the variance covariance matrix is
positive definite i.e. such that%
\begin{equation}
1+2\rho _{12}\rho _{13}\rho _{23}-\rho _{12}^{2}-\rho _{13}^{2}-\rho
_{23}^{2}>0.  \label{dod}
\end{equation}%
Then Kibble--Slepian formula reads that%
\begin{eqnarray*}
&&\exp \left( \frac{x_{1}^{2}+x_{2}^{2}+x_{3}^{2}}{2}\right) f_{3D}\left(
x_{1},x_{2},x_{3};\rho _{12},\rho _{13},\rho _{23}\right) \\
&=&\sum_{k,m,m=0}^{\infty }\frac{\rho _{12}^{k}\rho _{13}^{m}\rho _{23}^{n}}{%
k!m!n!}H_{k+m}\left( x_{1}\right) H_{k+n}\left( x_{2}\right) H_{m+n}\left(
x_{3}\right) .
\end{eqnarray*}

Thus immediate generalization of this formula would be to substitute the
Hermite polynomials by the $q-$Hermite ones and factorials by the $q-$%
factorials.

The question is if such sum is positive. It turns out that not in general
i.e. not for all $\rho _{12},$ $\rho _{13},$ $\rho _{23}$ satisfying (\ref%
{dod}). Nevertheless it is interesting to compute the sum 
\begin{equation}
\sum_{k,m,m=0}^{\infty }\frac{\rho _{12}^{k}\rho _{13}^{m}\rho _{23}^{n}}{%
\left[ k\right] _{q}!\left[ m\right] _{q}!\left[ n\right] _{q}!}%
H_{k+m}\left( x_{1}|q\right) H_{k+n}\left( x_{2}|q\right) H_{m+n}\left(
x_{3}|q\right) .  \label{suma}
\end{equation}%
For simplicity let us denote this sum by $g\left( x_{1},x_{2},x_{3}|\rho
_{12},\rho _{13},\rho _{23},q\right) .$

In \cite{Szab-KS} the following result have been formulated and proved.

\begin{theorem}
i)%
\begin{equation}
g\left( x_{1},x_{2},x_{3}|\rho _{12},\rho _{13},\rho _{23},q\right) =\frac{%
\left( \rho _{13}^{2}\right) _{\infty }}{\prod_{k=0}^{\infty }W_{q}\left(
x_{1},x_{3}|\rho _{13}q^{k}\right) }\sum_{s\geq 0}\frac{1}{\left[ s\right]
_{q}!}H_{s}\left( x_{2}|q\right) C_{s}\left( x_{1},x_{3}|\rho _{12},\rho
_{23},\rho _{13},q\right)  \notag
\end{equation}%
where $C_{n}\left( x_{1},x_{3}|\rho _{12},\rho _{23},\rho _{13},q\right) $
is given by either (\ref{_C}) or can be expressed in terms of polynomials $%
H_{n}$ in the following form: 
\begin{gather*}
C_{n}\left( x_{1},x_{3}|\rho _{12},\rho _{23},\rho _{13},q\right) =\frac{1}{%
\left( \rho _{13}^{2}\right) _{n}}\sum_{k=0}^{\left\lfloor n/2\right\rfloor
}(-1)^{k}q^{\binom{k}{2}}\QATOPD[ ] {n}{2k}_{q}\QATOPD[ ] {2k}{k}_{q}\left[ k%
\right] _{q}!\rho _{12}^{k}\rho _{13}^{k}\rho _{23}^{k}\left( \frac{\rho
_{12}\rho _{13}}{\rho _{23}}\right) _{k}\left( \frac{\rho _{13}\rho _{23}}{%
\rho _{12}}\right) _{k}\allowbreak \\
\sum_{j=0}^{n-2k}\QATOPD[ ] {n-2k}{j}_{q}\rho _{23}^{j}\left( \frac{\rho
_{12}\rho _{13}}{\rho _{23}}q^{k}\right) _{k}\rho _{12}^{n-j-2k}\left( \frac{%
\rho _{13}\rho _{23}}{\rho _{12}}q^{k}\right) _{n-2k-j}H_{j}\left(
x_{1}|q\right) H_{n-2k-j}\left( x_{3}|q\right) ,
\end{gather*}%
similarly for other pairs $\left( 1,3\right) $ and $\left( 2,3\right) ,$

ii) 
\begin{gather}
g\left( x_{1},x_{2},x_{3}|\rho _{12},\rho _{13},\rho _{23},q\right) =\frac{%
\left( \rho _{13}^{2},\rho _{23}^{2}\right) _{\infty }}{\prod_{k=0}^{\infty
}W_{q}\left( x_{1},x_{3}|\rho _{13}q^{k}\right) W_{q}(x_{3},x_{2}|\rho
_{23}q^{k})}  \label{exp_in_ASC} \\
\times \sum_{s=0}^{\infty }\frac{\rho _{12}^{s}\left( \rho _{13}\rho
_{23}/\rho _{12}\right) _{s}}{\left[ s\right] _{q}!\left( \rho
_{13}^{2}\right) _{s}\left( \rho _{23}^{2}\right) _{s}}P_{s}\left(
x_{1}|x_{3},\rho _{13},q\right) P_{s}\left( x_{2}|x_{3},\rho _{23},q\right) ,
\notag
\end{gather}%
similarly for other pairs $\left( 1,3\right) $ and $\left( 2,3\right) .$
\end{theorem}

Unfortunately as shown in \cite{Szab-KS}, one can find such $\rho _{12},$ $%
\rho _{13},$ $\rho _{23}$ that function $g$ with these parameters assumes
negative values for some $x_{j}\in S\left( q\right) $, $j\allowbreak
=\allowbreak 1,2,3$ hence consequently $g\left( x_{1},x_{2},x_{3}|\rho
_{12},\rho _{13},\rho _{23},q\right) \prod_{j=0}^{3}f_{N}\left(
x_{j}|q\right) $ with these values of parameters is not a density of a
probability distribution.

\begin{remark}
Notice that if $\rho _{12}\allowbreak =\allowbreak q^{m}\rho _{13}\rho _{23}$
then the sum in \ref{exp_in_ASC} is finite having only $m$ summands.
\end{remark}

\end{document}